\documentclass{amsart}

\usepackage{graphicx}
\usepackage[centertags]{amsmath}
\usepackage{amsfonts}
\usepackage{amssymb}
\usepackage{amsthm}
\usepackage{newlfont}
\usepackage{hyperref}
\usepackage{mathrsfs} 
\usepackage{yfonts} 
\usepackage{tensor} 

\newtheorem*{thm*}{Theorem}
\newtheorem{thm}{Theorem}[section]
\newtheorem{cor}[thm]{Corollary}
\newtheorem{lem}[thm]{Lemma}

\theoremstyle{definition}
\newtheorem{defn}[thm]{Definition}

\theoremstyle{remark}

\numberwithin{equation}{section}

\DeclareSymbolFont{bbold}{U}{bbold}{m}{n}
\DeclareSymbolFontAlphabet{\mathbbold}{bbold}

\newcommand{\N}{\mathbb{N}}
\newcommand{\Z}{\mathbb{Z}}
\newcommand{\R}{\mathbb{R}}
\newcommand{\id}{\mathrm{id}}
\newcommand{\acts}{\curvearrowright}
\newcommand{\Stab}{\mathrm{Stab}}
\newcommand{\Sym}{\mathrm{Sym}}
\newcommand{\pmp}{p{$.$}m{$.$}p{$.$}}

\newcommand{\cF}{\mathcal{F}}

\newcommand{\cP}{\mathcal{P}}
\newcommand{\cQ}{\mathcal{Q}}

\newcommand{\E}{\mathscr{E}}

\newcommand{\dist}{\mathrm{dist}}

\newcommand{\symd}{\triangle}
\newcommand{\salg}{\sigma \text{-}\mathrm{alg}}

\newcommand{\sH}{\mathrm{H}}
\newcommand{\dB}{d}
\newcommand{\dR}{d^{\mathrm{Rok}}}

\newcommand{\rh}{h}

\newcommand{\Borel}{\mathcal{B}}
\newcommand{\given}{\mathbin{|}}
\newcommand{\Given}{\mathbin{\Big|}}

\newcommand{\res}{\restriction}

\newcommand{\Prob}{\mathrm{Prob}}
\newcommand{\Aut}{\mathrm{Aut}}
\newcommand{\sjoin}{\mathrm{J}}
\newcommand{\sinv}{\mathscr{I}}
\newcommand{\Eavg}{\mathbb{E}}
\newcommand{\Var}{\mathrm{Var}}

\newcommand{\relprod}[1]{\times_{#1}}
\newcommand{\Sub}{\mathrm{Sub}}
\newcommand{\Cay}{\mathrm{Cay}}
\newcommand{\FMSF}{\mathrm{MSF}}
\newcommand{\Map}{\mathrm{Map}}

\begin{document}

\title[Weak containment and Rokhlin entropy]{Weak containment and Rokhlin entropy}
\author{Brandon Seward}
\address{Einstein Institute of Mathematics, The Hebrew University of Jerusalem, Givat Ram, Jerusalem 91904, Israel}
\email{b.m.seward@gmail.com}
\keywords{Rokhlin entropy, sofic entropy, weak containment, Pinsker algebra, Pinsker factor, subgroup formula}
\subjclass[2010]{37A35, 37A15}

\begin{abstract}
We define a new notion of weak containment for joinings, and we show that this notion implies an inequality between relative Rokhlin entropies. This leads to new upper bounds to Rokhlin entropy. We also use this notion to study how Pinsker algebras behave under direct products, and we study the Rokhlin entropy of restricted actions of finite-index subgroups.
\end{abstract}
\maketitle

\section{Introduction} \label{sec:intro}

We study the entropy theory of probability-measure-preserving ({\pmp}) actions of non-amenable groups. This research program was initiated through ground-breaking work of Bowen in 2008 \cite{B10b}. Bowen's work, combined with improvements by Kerr and Li \cite{KL11a}, created the notion of sofic entropy for {\pmp} actions of sofic groups which extends the classical notion of entropy for actions of amenable groups \cite{B12,KL13}. Among other things, this has led to the classification of many Bernoulli shifts over sofic groups up to isomorphism \cite{B10b,KL11b,B12b}. Drawing motivation from these developments, in \cite{S14} the author defined Rokhlin entropy for {\pmp} actions of general countable groups, in particular all non-amenable groups, which also extends the classical notion of entropy for free actions of amenable groups. While Rokhlin entropy is an upper bound to sofic entropy \cite{B10b,AS}, it is an open question if Rokhlin entropy and sofic entropy coincide for free actions of sofic groups when the sofic entropy is not minus infinity. It is also an open problem to compute the Rokhlin entropy for Bernoulli shifts over non-sofic groups \cite{S14a}.

We recall the definition of Rokhlin entropy. Let $G$ be a countable group, let $G \acts (X, \mu)$ be a (not necessarily free) {\pmp} action, and let $\sinv_G$ denote the $\sigma$-algebra of $G$-invariant sets. Write $\Borel(X)$ for the Borel $\sigma$-algebra of $X$. For a partition $\alpha$ of $X$, let $\salg_G(\alpha)$ denote the smallest $G$-invariant sub-$\sigma$-algebra containing $\alpha$. If $\cF$ is a $G$-invariant sub-$\sigma$-algebra, then the \emph{Rokhlin entropy of $G \acts (X, \mu)$ relative to $\cF$}, denoted $\rh_G(X, \mu \given \cF)$, is
$$\inf \Big\{ \sH(\alpha \given \cF \vee \sinv_G) : \alpha \text{ countable partition and } \salg_G(\alpha) \vee \cF \vee \sinv_G = \Borel(X) \Big\}.$$
The definition of conditional Shannon entropy $\sH(\cdot \given \cdot)$ is recalled in Section \ref{sec:pre}. When $\cF = \{\varnothing, X\}$ is trivial, then $\rh_G(X, \mu) = \rh_G(X, \mu \given \{\varnothing, X\})$ is called the \emph{Rokhlin entropy of $G \acts (X, \mu)$}. For free actions of amenable groups, Rokhlin entropy coincides with classical Kolmogorov--Sinai entropy \cite{AS}. For free ergodic actions of $\Z$ this is due to Rokhlin \cite{Ro67}.

In this paper we study Rokhlin entropy by using weak containment concepts. We first need a bit of notation. For a partition $\alpha$ of $X$ and a subset $T \subseteq G$ we write $\alpha^T = \bigvee_{t \in T} t^{-1} \cdot \alpha$, where $t^{-1} \cdot \alpha = \{t^{-1} \cdot A : A \in \alpha\}$. If $\alpha = \{A_1, \ldots, A_n\}$ is an ordered partition, then we write $\dist_\mu(\alpha)$ for the ordered tuple having $i^{\text{th}}$ coordinate $\mu(A_i)$. For $f \in \{1, \ldots, n\}^T$ set $A_f = \bigcap_{t \in T} t^{-1} \cdot A_{f(t)} \in \alpha^T$. By fixing an ordering on $G$ and applying the lexicographical order to $\{1, \ldots, n\}^T$, we obtain a canonical ordering of the partition $\alpha^T = \{A_f : f \in \{1, \ldots, n\}^T\}$. If $\alpha = \{A_1, \ldots, A_n\}$ and $\beta = \{B_1, \ldots, B_m\}$ are ordered partitions, then we similarly order $\alpha \vee \beta = \{C_{i,j} : 1 \leq i \leq n, \ 1 \leq j \leq m\}$ lexicographically, where $C_{i,j} = A_i \cap B_j$. We can now define the original notion of weak containment of actions as introduced by Kechris \cite{K12}. An action $G \acts (Z, \eta)$ \emph{weakly contains} another action $G \acts (Y, \nu)$ if for every finite ordered partition $\gamma$ of $Y$, every finite $T \subseteq G$, and every $\epsilon > 0$ there is an ordered partition $\zeta$ of $Z$ satisfying
$$|\dist_\eta(\zeta^T) - \dist_\nu(\gamma^T)| < \epsilon.$$
Here $| \cdot |$ denotes the $\ell^1$-norm. Weak containment can be equivalently defined by using the weak topology on the space of actions \cite{K12}. Two actions are \emph{weakly equivalent} if each weakly contains the other.

Now consider three {\pmp} actions: $G \acts (X, \mu)$, $G \acts (Y, \nu)$, and $G \acts (Z, \eta)$. Let $\lambda$ be a joining of $\mu$ with $\nu$ (i.e. a $G$-invariant probability measure on $X \times Y$ which has marginals $\mu$ and $\nu$), and let $\rho$ be a joining of $\mu$ with $\eta$. We say that $G \acts (X \times Z, \rho)$ \emph{weakly contains $G \acts (X \times Y, \lambda)$ as joinings with $G \acts (X, \mu)$} if for every finite ordered partition $\alpha$ of $X$, every finite ordered partition $\gamma$ of $Y$, every finite $T \subseteq G$, and every $\epsilon > 0$ there is a finite ordered partition $\zeta$ of $Z$ satisfying
$$|\dist_\rho(\alpha \vee \zeta^T) - \dist_\lambda(\alpha \vee \gamma^T)| < \epsilon.$$
It is immediately seen that when $\lambda = \mu \times \nu$ and $\rho = \mu \times \eta$, we have $G \acts (X \times Z, \rho)$ weakly contains $G \acts (X \times Y, \lambda)$ as joinings with $G \acts (X, \mu)$ if and only if $G \acts (Z, \eta)$ weakly contains $G \acts (Y, \nu)$.

This new notion of weak containment for joinings is discussed in greater detail in Sections \ref{sec:join} and \ref{sec:stab}. A well known theorem of Ab\'{e}rt and Weiss states that Bernoulli shift actions are weakly contained in all free actions \cite{AW13}. This result was extended to non-free actions by Tucker-Drob \cite{TD12}. We prove an analogous result for joinings in Section \ref{sec:bs}; see Lemma \ref{lem:awcent}.

Our study of Rokhlin entropy is based off of the following important lemma. Recall that a {\pmp} action $G \acts (X, \mu)$ is \emph{aperiodic} if $\mu$-almost-every orbit is infinite.

\begin{lem} \label{intro:lem}
Let $G \acts (X \times Y, \lambda)$ and $G \acts (X \times Z, \rho)$ be joinings with an aperiodic action $G \acts (X, \mu)$. Let $\cF$ be a $G$-invariant sub-$\sigma$-algebra of $X$. If $G \acts (X \times Z, \rho)$ weakly contains $G \acts (X \times Y, \lambda)$ as joinings with $G \acts (X, \mu)$, then
$$\rh_G(X \times Z, \rho \given \cF \vee \Borel(Z)) \leq \rh_G(X \times Y, \lambda \given \cF \vee \Borel(Y)).$$
\end{lem}

A natural conjecture is that in the case of direct product joinings we have $\rh_G(X \times Y, \mu \times \nu \given \Borel(Y)) = \rh_G(X, \mu)$ for all free actions $G \acts (X, \mu)$ and $G \acts (Y, \nu)$. This is known to hold when $G$ is amenable but is unknown otherwise. However, by using the above lemma we show that this equality holds under a weak containment assumption.

\begin{thm} \label{intro:prod}
Let $G$ be a countably infinite group, let $G \acts (X, \mu)$ be a free {\pmp} action, and let $\cF$ be a $G$-invariant sub-$\sigma$-algebra. If $G \acts (Y, \nu)$ is a {\pmp} action which is weakly contained in all free {\pmp} actions of $G$ then
$$\rh_G(X, \mu \given \cF) = \rh_G(X \times Y, \mu \times \nu \given \cF \vee \Borel(Y)).$$
\end{thm}

We remark that a related but more difficult problem asks if $\rh_G(X \times Y, \mu \times \nu) = \rh_G(X, \mu) + \rh_G(Y, \nu)$ for free actions $G \acts (X, \mu)$ and $G \acts (Y, \nu)$. This is open for both Rokhlin and sofic entropy when $G$ is non-amenable, but Austin has made good progress on this problem for sofic entropy \cite{A15}.

For an amenable group $G$, every pair of free actions are weakly equivalent \cite{K12}. Thus the above theorem recovers what is known in the amenable case.

A particular instance of the above theorem is when $G \acts (Y, \nu)$ is a Bernoulli shift \cite{AW13}.

\begin{cor} \label{intro:bern}
Let $G$ be a countably infinite group, let $G \acts (X, \mu)$ be a free {\pmp} action, and let $\cF$ be a $G$-invariant sub-$\sigma$-algebra. For every standard probability space $(L, \lambda)$ we have
$$\rh_G(X, \mu \given \cF) = \rh_G(X \times L^G, \mu \times \lambda^G \given \cF \vee \Borel(L^G)).$$
\end{cor}

The value of this corollary is that the right-hand side is a bit more manageable and leads to new upper bounds to Rokhlin entropy (and thus upper bounds to sofic entropy as well). Using this corollary, we deduce Theorems \ref{intro:int} and \ref{intro:drop} below. We regard these two theorems to be the most important results of the paper.

In the special case of sofic entropy, the upper bound appearing below was obtained independently by both Andrei Alpeev and Lewis Bowen (personal communication).

\begin{thm} \label{intro:int}
Let $G$ be a countably infinite group, let $G \acts (X, \mu)$ be a free {\pmp} action, and let $\cF$ be a $G$-invariant sub-$\sigma$-algebra. Consider the Bernoulli shift $([0, 1]^G, \lambda^G)$ where $\lambda$ is Lebesgue measure. For $y \in [0,1]^G$ set $L_y = \{g \in G : y(g^{-1}) < y(1_G)\}$. If $\alpha$ is a partition with $\sH(\alpha \given \cF \vee \sinv_G) < \infty$ and $\salg_G(\alpha) \vee \cF \vee \sinv_G = \Borel(X)$ then
$$\rh_G(X, \mu \given \cF) \leq \int_{[0,1]^G} \sH_\mu(\alpha \given \alpha^{L_y} \vee \cF \vee \sinv_G) \ d \lambda^G(y).$$
\end{thm}

Intuitively, one should view the above sets $L_y \subseteq G$ as providing a randomized past for the action of $G$ on $(X, \mu)$.

The next upper bound improves \cite[Theorem 1.3]{S14a}. In the special case of sofic entropy, the upper bound appearing below was independently obtained by Mikl\'{o}s Ab\'{e}rt, Tim Austin, Lewis Bowen, and Benjamin Weiss (personal communication). Peter Burton also independently obtained a related version of this upper bound for topological sofic entropy \cite{Bu15}.

\begin{thm} \label{intro:drop}
Let $G$ be a countably infinite group, let $G \acts (X, \mu)$ be a free {\pmp} action, and let $\cF$ be a $G$-invariant sub-$\sigma$-algebra. If $\alpha$ is any countable partition with $\salg_G(\alpha) \vee \cF \vee \sinv_G = \Borel(X)$ then
$$\rh_G(X, \mu \given \cF) \leq \inf_{\substack{T \subseteq G\\T \text{ finite}}} \frac{1}{|T|} \cdot \sH(\alpha^T \given \cF \vee \sinv_G).$$
\end{thm}

The upper bounds in Theorems \ref{intro:int} and \ref{intro:drop} are optimal in the sense that for free actions of amenable groups $G$ the expressions coincide with classical entropy (equivalently Rokhlin entropy). For the expression in Theorem \ref{intro:int} this was proven by Kieffer \cite[Theorem 3]{Ki}. For the expression in Theorem \ref{intro:drop} this is a folklore fact which has appeared in \cite{DGRS,DFR}. However we mention that Theorem \ref{intro:drop} leads to a new proof that the right-hand expression in that theorem coincides with classical entropy when $G$ is amenable and the action is free (by definition the classical entropy is equal to the right-hand side when $T$ is restricted to a sequence of F{\o}lner sets; thus classical entropy is greater than or equal to the expression in Theorem \ref{intro:drop}, but we know that Rokhlin entropy and classical entropy coincide for free actions of amenable groups \cite{AS}). Unfortunately, for non-amenable groups these expressions do not coincide with Rokhlin entropy and are not even isomorphism invariants. For the expression in Theorem \ref{intro:int}, we prove this in Lemma \ref{lem:fail}. For the expression in Theorem \ref{intro:drop} this is due to Bowen and is recorded in \cite{Bu15}.

Using our new notion of weak containment of joinings and Lemma \ref{intro:lem}, we explore two additional topics in Rokhlin entropy theory. The first is the validity of the ``subgroup formula.'' This conjectured formula states that if $G \acts (X, \mu)$ is a free {\pmp} action and $\Gamma \leq G$ is a finite-index subgroup then $\rh_\Gamma(X, \mu) = |G : \Gamma| \cdot \rh_G(X, \mu)$. This is known to hold when $G$ is amenable (it follows from Theorems \ref{intro:prod} and \ref{intro:sub} here and can also be found in \cite[Theorem 2.16]{Da01}). For non-amenable groups it is unknown, both for sofic and Rokhlin entropy, but holds for a related quantity called the f-invariant \cite{S12a}. We find that this question is related to the earlier question of whether $\rh_G(X \times Y, \mu \times \nu \given \Borel(Y)) = \rh_G(X, \mu)$.

Recall that an action $G \acts (X, \mu)$ is \emph{finite} if there is a normal finite-index subgroup $\Delta \leq G$ such that $\Delta$ fixes every point in $X$. More generally, an action is called \emph{finitely modular} if it is an inverse limit of finite actions.

\begin{thm} \label{intro:sub}
Let $G \acts (X, \mu)$ be an aperiodic {\pmp} action, let $\cF$ be a $G$-invariant sub-$\sigma$-algebra, and assume $\rh_G(X, \mu \given \cF) < \infty$. The following are equivalent.
\begin{enumerate}
\item[\rm (i)] $\rh_G(X, \mu \given \cF) = \rh_G(X \times Y, \mu \times \nu \given \cF \vee \Borel(Y))$ for every finitely modular action $G \acts (Y, \nu)$.
\item[\rm (ii)] $\rh_\Gamma(X, \mu \given \cF) = |G : \Gamma| \cdot \rh_G(X, \mu \given \cF)$ for every finite-index subgroup $\Gamma \leq G$.
\end{enumerate}
\end{thm}

We also relate the subgroup formula to the Rokhlin versus sofic entropy problem. Below, for a sofic approximation $\Sigma$ to $G$, we write $h_G^\Sigma$ for the corresponding sofic entropy (see Section \ref{sec:pre} for definitions). Also, for a finite set $S$ we write $u_S$ for the normalized counting measure on $S$.

\begin{thm} \label{intro:sofic}
Let $G$ be a sofic group with sofic approximation $\Sigma$, and let $G \acts (X, \mu)$ be an aperiodic {\pmp} action. Assume that for all finite-index normal subgroups $\Delta \lhd G$ we have $h_G^\Sigma(G / \Delta, u_{G / \Delta}) \neq - \infty$. Then
\begin{enumerate}
\item[\rm (i)] for every finitely modular action $G \acts (Y, \nu)$
$$h_G^\Sigma(X, \mu) \leq \rh_G(X \times Y, \mu \times \nu \given \Borel(Y)) \leq \rh_G(X, \mu)$$
\item[\rm (ii)] for every finite-index subgroup $\Gamma \leq G$
$$|G : \Gamma| \cdot h_G^\Sigma(X, \mu) \leq \rh_\Gamma(X, \mu) \leq |G : \Gamma| \cdot \rh_G(X, \mu)$$
\end{enumerate}
\end{thm}

When combined with Lemma \ref{intro:lem}, the previous two theorems take a stronger form when $G$ has property MD. Recall that a countable residually finite group $G$ is said to have \emph{property MD} if there is a finitely modular action which weakly contains all other {\pmp} actions of $G$ \cite{K12}. It is known that all residually finite amenable groups, all free groups, all free products of finite groups, all surface groups, and all fundamental groups of closed hyperbolic $3$-manifolds have property MD \cite{K12,BTD,Ag13}. Also, property MD is preserved under passage to subgroups and extensions by residually finite amenable groups \cite{K12,BTD}.

The important feature of the following corollary is that it discusses product actions $G \acts (X \times Y, \mu \times \nu)$ where $G \acts (Y, \nu)$ varies over all {\pmp} actions, rather than only the finitely modular actions.

\begin{cor}
Let $G$ be a residually finite group with property MD, let $G \acts (X, \mu)$ be an aperiodic {\pmp} action, and let $\cF$ be a $G$-invariant sub-$\sigma$-algebra with $\rh_G(X, \mu \given \cF) < \infty$. The following are equivalent.
\begin{enumerate}
\item[\rm (1)] $\rh_G(X, \mu \given \cF) = \rh_G(X \times Y, \mu \times \nu \given \cF \vee \Borel(Y))$ for all {\pmp} actions $G \acts (Y, \nu)$.
\item[\rm (2)] $\rh_\Gamma(X, \mu \given \cF) = |G : \Gamma| \cdot \rh_G(X, \mu \given \cF)$ for every finite-index subgroup $\Gamma \leq G$.
\end{enumerate}
Furthermore, if $\Sigma$ is a sofic approximation to $G$ with $h_G^\Sigma(G / \Delta, u_{G / \Delta}) \neq - \infty$ for every finite-index normal subgroup $\Delta \lhd G$ and $h_G^\Sigma(X, \mu) = \rh_G(X, \mu) < \infty$, then (1) and (2) hold with $\cF = \{\varnothing, X\}$.
\end{cor}

The final topic we consider is outer Pinsker algebras. Recall that for a {\pmp} action $G \acts (X, \mu)$ of an amenable group $G$, the Pinsker algebra is defined to be the largest $G$-invariant sub-$\sigma$-algebra of $X$ for which the corresponding factor, called the Pinsker factor, has entropy $0$. This definition still makes sense for actions of non-amenable groups and leads to the concept of a (Rokhlin) Pinsker algebra. However, the corresponding Pinsker factor is a bit strange as it may admit factors of positive entropy (since entropy can increase under factor maps for actions of non-amenable groups). There is an alternate notion which in some ways behaves better. For this, we recall the definition of outer Rokhlin entropy from \cite{S14a}. If $G \acts (X, \mu)$ is a {\pmp} action, $\cF$ is a $G$-invariant sub-$\sigma$-algebra, and $\mathcal{C} \subseteq \Borel(X)$, then the \emph{outer Rokhlin entropy of $\mathcal{C}$ relative to $\cF$}, denoted $\rh_{G,\mu}(\mathcal{C} \given \cF)$, is
$$\inf \Big\{ \sH(\alpha \given \cF \vee \sinv_G) : \alpha \text{ countable partition and } \mathcal{C} \subseteq \salg_G(\alpha) \vee \cF \vee \sinv_G \Big\}.$$
If $G \acts (Y, \nu)$ is a factor of $G \acts (X, \mu)$ via $f : X \rightarrow Y$, then we define $\rh_{G, \mu}(Y, \nu) = \rh_{G, \mu}(f^{-1}(\Borel(Y)))$. Note that
$$\rh_{G,\mu}(Y, \nu) \leq \min(\rh_G(Y, \nu), \rh_G(X, \mu)).$$
The \emph{outer (Rokhlin) Pinsker algebra of $G \acts (X, \mu)$ relative to $\cF$}, denoted $\Pi(\mu \given \cF)$, is defined to be the largest $G$-invariant sub-$\sigma$-algebra for which $\rh_{G,\mu}(\Pi(\mu \given \cF) \given \cF) = 0$. Note that $\cF \subseteq \Pi(\mu \given \cF)$. A similar notion of outer sofic Pinsker algebra was introduced by Hayes in \cite{H15}.

Using weak containment concepts, we study how outer Pinsker algebras behave for direct products.

We remark that Ben Hayes has obtained sofic entropy versions of Theorem \ref{intro:pinsk} and Corollary \ref{intro:pinsk2} below (personal communication).

\begin{thm} \label{intro:pinsk}
Let $G$ be a countably infinite group, let $G \acts (X, \mu)$ be a free {\pmp} action, and let $\cF$ be a $G$-invariant sub-$\sigma$-algebra. If $G \acts (Y, \nu)$ is a {\pmp} action which is weakly contained in all free {\pmp} actions of $G$, then
$$\Pi(\mu \times \nu \given \cF \vee \Borel(Y)) = \Pi(\mu \given \cF) \vee \Borel(Y).$$
\end{thm}

For actions of amenable groups, it is well known that the Pinsker algebra of a direct product action is the join of the Pinsker algebras of the two factors \cite{GTW}. It is unknown if this property holds for sofic entropy or Rokhlin entropy. Under the assumption that both actions are free and weakly contained in all free actions, we prove this holds for Rokhlin entropy. Again, since all free actions of an amenable group $G$ are weakly equivalent, this recovers what is known for amenable groups.

\begin{cor} \label{intro:pinsk2}
Let $G$ be a countably infinite group, let $G \acts (X, \mu)$ and $G \acts (Y, \nu)$ be free {\pmp} actions which are weakly contained in all free {\pmp} actions of $G$. Let $\cF$ and $\Sigma$ be $G$-invariant sub-$\sigma$-algebras of $X$ and $Y$, respectively. Then
$$\Pi(\mu \times \nu \given \cF \vee \Sigma) = \Pi(\mu \given \cF) \vee \Pi(\nu \given \Sigma).$$
\end{cor}

We remark that Theorems \ref{intro:prod}, \ref{intro:int}, \ref{intro:pinsk} and Corollaries \ref{intro:bern} and \ref{intro:pinsk2} are stated for free actions for simplicity. We prove these results for actions which are not necessarily free. See Theorem \ref{thm:prod}, Corollary \ref{cor:int}, Theorem \ref{thm:cpeprod}, and Corollaries \ref{cor:prod} and \ref{cor:pinsk2}, respectively.

\vspace{3mm}
\noindent\emph{Acknowledgments.}
This research was partially supported by NSF RTG grant 1045119 and ERC grant 306494. The author is thankful for valuable conversations with Mikl\'{o}s Ab\'{e}rt, Andrei Alpeev, Tim Austin, Lewis Bowen, Damien Gaboriau, Ben Hayes, Mike Hochman, Russell Lyons, Ralf Spatzier, and Benjy Weiss.

\section{Preliminaries} \label{sec:pre}

Throughout this paper probability spaces $(X, \mu)$ will always be assumed to be standard, meaning that $X$ is a standard Borel space and $\mu$ is a Borel probability measure. If $f : X \rightarrow Y$ is a Borel map, then we write $f_*(\mu)$ for the push-forward measure. If $X$ is a set and $x \in X$, we write $\delta_x$ for the Borel probability measure of $X$ which is supported on the singleton $\{x\}$. If $X$ is finite then we write $u_X$ for the normalized counting measure on $X$. For $n \in \N$ we identify $n$ with the set $\{0, 1, \ldots, n - 1\}$. Thus $2^G = \{0, 1\}^G$ and $u_2^G = u_{\{0,1\}}^G$.

We write $\Borel(X)$ for the $\sigma$-algebra of Borel subsets of $X$. A sub-$\sigma$-algebra $\cF$ of $X$ is \emph{countably generated} if there is a countable algebra $\mathcal{A}$ such that $\cF$ is the smallest $\sigma$-algebra containing $\mathcal{A}$. For a standard Borel space $X$, it is well known that $\Borel(X)$ is countably generated. Furthermore, if $\mu$ is a Borel probability measure on $X$ and $\cF \subseteq \Borel(X)$ is a sub-$\sigma$-algebra, then it is well known that there is a countably generated $\sigma$-algebra $\cF' \subseteq \Borel(X)$ which coincides with $\cF$ up to $\mu$-null sets. We will frequently ignore null sets without mention. For instance, we will write $A = B$ if $A, B \subseteq X$ have null symmetric difference $\mu(A \symd B) = 0$.

For a product $X \times Y$, we write $\pi^X$ and $\pi^Y$ for the coordinate projection maps. If $X$ and $Y$ are standard Borel spaces, then we will naturally view $\Borel(X)$ and $\Borel(Y)$ as subsets of $\Borel(X \times Y)$ via their pre-images under $\pi^X$ and $\pi^Y$. Therefore if $A \subseteq X$ and $B \subseteq Y$ are Borel then we will write $A \cap B$ for the set $A \times B \subseteq X \times Y$. If $(X, \mu)$, $(Y, \nu)$, and $(Z, \eta)$ are probability spaces and $p : (X, \mu) \rightarrow (Z, \eta)$ and $q : (Y, \nu) \rightarrow (Z, \eta)$ are measure-preserving maps, then the \emph{relatively independent coupling} of $\mu$ and $\nu$ over $\eta$ is the measure $\mu \relprod{\eta} \nu$ on $X \times Y$ defined by
$$\mu \relprod{\eta} \nu = \int_Z \mu_z \times \nu_z \ d \eta(z),$$
where $\mu = \int_Z \mu_z \ d \eta(z)$ and $\nu = \int_Z \nu_z \ d \eta(z)$ are the disintegrations of $\mu$ and $\nu$ over $\eta$ (as given by the maps $p$ and $q$). The measure $\mu \relprod{\eta} \nu$ depends on the maps $p$ and $q$, but this dependence is omitted from the notation. When a group $G$ acts on each of $(X, \mu)$, $(Y, \nu)$, and $(Z, \eta)$ and the maps $p$ and $q$ are $G$-equivariant, we call $\mu \relprod{\eta} \nu$ the \emph{relatively independent joining} of $\mu$ and $\nu$ over $\eta$.

If two labeled partitions $\alpha = \{A_i : i \in I\}$ and $\beta = \{B_i : i \in I\}$ of $(X, \mu)$ have the same set of labels, then we define $\dB_\mu(\alpha, \beta) = \sum_{i \in I} \mu(A_i \symd B_i)$. For two (possibly unlabeled) partitions $\alpha$ and $\beta$, we write $\alpha \geq \beta$ if $\alpha$ is finer than $\beta$. The \emph{Shannon entropy} of a countable partition $\alpha$ of $(X, \mu)$ is $\sH(\alpha) = \sum_{A \in \alpha} - \mu(A) \cdot \log(\mu(A))$. We write $\sH_\mu(\alpha)$ when we wish to emphasize the measure. For a sub-$\sigma$-algebra $\cF \subseteq \Borel(X)$, let $f : (X, \mu) \rightarrow (Y, \nu)$ be the associated factor, and let $\mu = \int_Y \mu_y \ d \nu(y)$ be the corresponding disintegration of $\mu$ over $\nu$. If $\alpha$ is a countable partition of $X$ then the \emph{conditional Shannon entropy} of $\alpha$ relative to $\cF$ is
$$\sH(\alpha \given \cF) = \int_Y \sH_{\mu_y}(\alpha) \ d \nu(y).$$
If $\beta$ is a partition then we define $\sH(\alpha \given \beta) = \sH(\alpha \given \cF)$ where $\cF$ is the $\sigma$-algebra generated by $\beta$. A simple exercise shows that if $\alpha$ and $\beta$ are countable partitions and $\cF$ is a sub-$\sigma$-algebra with $\beta \subseteq \cF$ then
$$\sH(\alpha \given \cF) = \sum_{B \in \beta} \mu(B) \cdot \sH_B(\alpha \given \cF),$$
where we write $\sH_B$ for $\sH_{\mu_B}$ where $\mu_B$ is the measure defined by $\mu_B(C) = \mu(B \cap C) / \mu(B)$.

We recall some well known properties of Shannon entropy.

\begin{lem}[see \cite{Do11}] \label{lem:shan}
Let $(X, \mu)$ be a standard probability space, let $\alpha$ and $\beta$ be countable Borel partitions of $X$, and let $\cF$, $\Sigma$, and $(\cF_n)_{n \in \N}$ be sub-$\sigma$-algebras. Then
\begin{enumerate}
\item[\rm (i)] $\sH(\alpha \given \cF) = 0$ if and only if $\alpha \subseteq \cF$ mod null sets;
\item[\rm (ii)] if $\sH(\alpha) < \infty$ then $\sH(\alpha \given \cF) = \sH(\alpha)$ if and only if $\alpha$ and $\cF$ are independent;
\item[\rm (iii)] $\sH(\alpha \vee \beta \given \cF) = \sH(\beta \given \cF) + \sH(\alpha \given \beta \vee \cF)$;
\item[\rm (iv)] $\sH(\alpha \given \cF)$ equals the supremum of $\sH(\beta \given \cF)$ over finite partitions $\beta \leq \alpha$;
\item[\rm (v)] if $\Sigma \subseteq \cF$ then $\sH(\alpha \given \Sigma) \geq \sH(\alpha \given \cF)$;
\item[\rm (vi)] $\sH(\alpha \given \bigvee_{n \in \N} \cF_n) = \inf_{n \in \N} \sH(\alpha \given \cF_n)$ if the $\cF_n$'s are increasing and the right-hand side is finite;
\item[\rm (vii)] $\sH(\alpha \given \bigcap_{n \in \N} \cF_n) = \lim_{n \rightarrow \infty} \sH(\alpha \given \cF_n)$ if the $\cF_n$'s are decreasing and $\sH(\alpha) < \infty$.
\end{enumerate}
\end{lem}

The next lemma illustrates a useful property of countably generated $\sigma$-algebras.

\begin{lem}
Let $(X, \mu)$ be a standard probability space, let $\Sigma \subseteq \Borel(X)$ be a sub-$\sigma$-algebra, let $f : (X, \mu) \rightarrow (Y, \nu)$ be the factor associated to $\Sigma$, and let $\mu = \int_Y \mu_y \ d \nu(y)$ be the corresponding disintegration of $\mu$ over $\nu$. If $\cF \subseteq \Borel(X)$ is a countably generated $\sigma$-algebra, then for every countable partition $\alpha$
$$\sH(\alpha \given \cF \vee \Sigma) = \int_Y \sH_{\mu_y}(\alpha \given \cF) \ d \nu(y).$$
\end{lem}

\begin{proof}
By the monotone convergence theorem and Lemma \ref{lem:shan}.(iv) it suffices to prove this for finite partitions $\alpha$. Since $\cF$ is countably generated, there is an increasing sequence of finite partitions $\beta_n$ such that $\cF$ is the smallest $\sigma$-algebra containing every $\beta_n$. By Lemma \ref{lem:shan}.(vi), for every Borel probability measure $\eta$ on $X$ and every finite partition $\alpha$ we have $\sH_\eta(\alpha \given \cF) = \inf_n \sH_\eta(\alpha \given \beta_n)$. From Lemma \ref{lem:shan}.(iii) one can deduce that $\sH_\mu(\alpha \given \beta_n \vee \Sigma) = \int_Y \sH_{\mu_y}(\alpha \given \beta_n) \ d \nu(y)$. Now apply the monotone convergence theorem.
\end{proof}

The set of all countable partitions $\alpha$ with $\sH(\alpha) < \infty$ becomes a complete metric space under the \emph{Rokhlin metric} $\dR_\mu$ defined by
$$\dR_\mu(\alpha, \beta) = \sH(\alpha \given \beta) + \sH(\beta \given \alpha).$$
It is known that for every $k \in \N$, $\dR_\mu(\alpha, \beta)$ is uniformly bounded above by $\dB_\mu(\alpha, \beta)$ for $k$-piece labeled partitions $\alpha$ and $\beta$ \cite[Fact 1.7.7]{Do11}. It is not difficult to check that $|\sH(\alpha) - \sH(\beta)| \leq \dR_\mu(\alpha, \beta)$ and that $|\sH(\gamma \given \alpha) - \sH(\gamma \given \beta)| \leq 2 \cdot \dR_\mu(\alpha, \beta)$. Also, if $G \acts (X, \mu)$ is a {\pmp} action and $T \subseteq G$ is finite then $\dR_\mu(\alpha^T, \beta^T) \leq |T| \cdot \dR_\mu(\alpha, \beta)$. Proofs of these facts can be found in the appendix to \cite{S14a}.

Next we recall some results on Rokhlin entropy which we will need.

\begin{thm}[Seward--Tucker-Drob, in preparation (see \cite{ST14} for free actions)] \label{thm:robin}
Let $G$ be a countable group and let $G \acts (X, \mu)$ be an aperiodic {\pmp} action. Then for every $\epsilon > 0$ there is a factor $G \acts (Y, \nu)$ of $(X, \mu)$, say via $f : (X, \mu) \rightarrow (Y, \nu)$, such that $\rh_G(Y, \nu) < \epsilon$ and $\Stab(f(x)) = \Stab(x)$ for every $x \in X$.
\end{thm}

The following is an important and quite useful property of Rokhlin entropy. We call this the \emph{sub-additivity property} of Rokhlin entropy.

\begin{lem}[Alpeev--Seward \cite{AS}] \label{lem:add2}
Let $G \acts (X, \mu)$ be a {\pmp} action, let $\mathcal{C} \subseteq \Borel(X)$, let $\Sigma$ be a $G$-invariant sub-$\sigma$-algebra, and let $(\cF_n)_{n \in \N}$ be an increasing sequence of $G$-invariant sub-$\sigma$-algebras with $\mathcal{C} \subseteq \bigvee_{n \in \N} \cF_n \vee \Sigma$. Then
\begin{equation} \label{eqn:add2}
\rh_{G,\mu}(\mathcal{C} \given \Sigma) \leq \rh_{G,\mu}(\cF_1 \given \Sigma) + \sum_{n \geq 2} \rh_{G,\mu}(\cF_n \given \cF_{n-1} \vee \Sigma).
\end{equation}
\end{lem}

The last result we need is a formula for the Rokhlin entropy of an inverse limit of actions.

\begin{thm}[Alpeev--Seward \cite{AS}] \label{thm:ks}
Let $G \acts (X, \mu)$ be a {\pmp} action and let $\cF$ be a $G$-invariant sub-$\sigma$-algebra. Suppose that $G \acts (X, \mu)$ is the inverse limit of actions $G \acts (X_n, \mu_n)$. Identify each $\Borel(X_n)$ as a sub-$\sigma$-algebra of $X$ in the natural way. Then
\begin{equation*}
\rh_G(X, \mu \given \cF) < \infty \Longleftrightarrow
\left\{\begin{array}{c}
\displaystyle{\inf_{n \in \N} \sup_{m \geq n} \rh_{G,\mu}(\Borel(X_m) \given \Borel(X_n) \vee \cF) = 0}\\
\displaystyle{\text{and} \quad \forall m \ \rh_{G,\mu}(\Borel(X_m) \given \cF) < \infty.}
\end{array}\right\}\end{equation*}
Furthermore, when $\rh_G(X, \mu \given \cF) < \infty$ we have
\begin{equation*}
\rh_G(X, \mu \given \cF) = \sup_{m \in \N} \rh_{G,\mu}(\Borel(X_m) \given \cF).
\end{equation*}
\end{thm}

Finally, we briefly review the definition of sofic entropy. A \emph{sofic approximation} to a group $G$ is a sequence $(d_n)_{n \in \N}$ of integers with $\lim_{n \rightarrow \infty} d_n = \infty$ and a sequence of maps (not necessarily homomorphisms) $\sigma_n : G \rightarrow \Sym(d_n)$ such that
\begin{enumerate}
\item[\rm (1)] for all $g, h \in G$, $\lim_{n \rightarrow \infty} \frac{1}{d_n} \cdot | \{0 \leq i < d_n : \sigma_n(g) \circ \sigma_n(h)(i) = \sigma_n(g h)(i)\}| = 1$, and
\item[\rm (2)] for all $1_G \neq g \in G$, $\lim_{n \rightarrow \infty} \frac{1}{d_n} \cdot |\{0 \leq i < d_n : \sigma_n(g)(i) \neq i\}| = 1$.
\end{enumerate}
The group $G$ is called \emph{sofic} if it admits a sofic approximation $\Sigma$. The class of sofic groups contains the class of countable amenable groups and the class of residually finite groups. It is a well known open problem to determine if all countable groups are sofic. For a survey of sofic groups, see \cite{Pe08}.

Now let $G$ be a sofic group with sofic approximation $\Sigma = (\sigma_n : G \rightarrow \Sym(d_n))_{n \in \N}$, and let $G \acts (X, \mu)$ be a {\pmp} action. We will use the notation of \cite[Definition 3.3]{KL13} in order to define the sofic entropy $h_G^\Sigma(X, \mu)$. It is well known that without loss of generality, we may assume that $X$ is a compact metric space with metric $\rho$ and that $G$ acts continuously on $X$. For two maps $\phi, \psi : \{0, \ldots, d-1\} \rightarrow X$ define
$$\rho_2(\phi, \psi) = \Bigg( \frac{1}{d} \cdot \sum_{i = 0}^{d-1} \rho(\phi(i), \psi(i))^2 \Bigg)^{1/2}.$$
For $Y \subseteq X^{\{0, \ldots, d-1\}}$ and $\epsilon > 0$ write $N_\epsilon(Y, \rho_2)$ for the maximum cardinality of sets $A \subseteq Y$ consisting of points which are pairwise $\rho_2$-distance at least $\epsilon$ apart. Let $\Prob(X)$ denote the compact space of all Borel probability measures on $X$, equipped with the weak$^*$-topology. For an open neighborhood $U \subseteq \Prob(X)$ of $\mu$, finite $T \subseteq G$, and $\delta > 0$, define $\Map(\rho, T, U, \delta, \sigma_n)$ to be the set of all maps $\phi : \{0, \ldots, d_n-1\} \rightarrow X$ satisfying
\begin{enumerate}
\item[\rm (1)] $\rho_2(\phi \circ \sigma_n(t), t \cdot \phi) < \delta$ for all $t \in T$, and
\item[\rm (2)] $\phi_*(u_{d_n}) \in U$.
\end{enumerate}
For $\epsilon > 0$ define
$$h_G^{\Sigma, \epsilon}(X, \mu) = \inf_T \inf_U \inf_{\delta > 0} \limsup_{n \rightarrow \infty} \frac{1}{d_n} \cdot \log N_\epsilon(\Map(\rho, T, U, \delta, \sigma_n), \rho_2),$$
where $T$ ranges over all finite subsets of $G$ and $U$ ranges over all open neighborhoods of $\mu$. The \emph{$\Sigma$-sofic entropy of $G \acts (X, \mu)$} is then defined to be $h_G^\Sigma(X, \mu) = \sup_{\epsilon > 0} h_G^{\Sigma, \epsilon}(X, \mu)$. Note that $h_G^\Sigma(X, \mu) \in \{- \infty\} \cup [0, + \infty]$, with the case $h_G^\Sigma(X, \mu) = - \infty$ occurring when $\Map(\rho, T, U, \delta, \sigma_n)$ is empty for some $T$, $U$, and $\delta$ and all sufficiently large $n$.

\section{Weak containment of joinings} \label{sec:join}

For a group homomorphism $a : G \rightarrow \Aut(X, \mu)$, we write $G \acts^a (X, \mu)$ for the corresponding {\pmp} action. Recall that if $G \acts^a (X, \mu)$ and $G \acts^b (Y, \nu)$ are {\pmp} actions, then a \emph{joining} of $a$ with $b$ is an $a \times b$-invariant probability measure $\lambda$ on $X \times Y$ which has marginals $\mu$ and $\nu$ on $X$ and $Y$, respectively. We will typically view $a$ as fixed while both $G \acts^b (Y, \nu)$ and $\lambda$ vary. Note that from a joining $\lambda$ one can recover $\nu$ by projecting to $Y$, and the particular choice of $Y$ is unimportant since all uncountable standard Borel spaces are Borel isomorphic. Therefore, for a fixed {\pmp} action $G \acts^a (X, \mu)$, we can fix an uncountable standard Borel space $Y$ and then describe all possible joinings with $a$ by pairs $(b, \lambda)$ where $G \acts^b Y$ is a Borel action and $\lambda$ is an $a \times b$-invariant probability measure on $X \times Y$ whose marginal on $X$ is $\mu$ (note the marginal on $Y$ is automatically $b$-invariant).

\begin{defn} \label{defn:wcj}
Fix a {\pmp} action $G \acts^a (X, \mu)$ and fix an uncountable standard Borel space $Y$. We define the \emph{space of joinings with $a$}, $\sjoin(a)$, to be the set of all pairs $(b, \lambda)$ where $G \acts^b Y$ is a Borel action and $\lambda$ is a $a \times b$-invariant Borel probability measure on $X \times Y$ with $\pi^X_*(\lambda) = \mu$. We topologize $\sjoin(a)$ as follows. For $(b_2, \lambda_2) \in \sjoin(a)$, a neighborhood base for this point is given by the sets
\begin{equation*}
\left\{(b_1, \lambda_1) :
\begin{array}{l}
\exists \text{ ordered partition } \gamma_1 \text{ of } Y \text{ with}\\
\quad\left|\dist_{\lambda_1}(\cP \vee \gamma_1^{b_1(T)}) - \dist_{\lambda_2}(\cP \vee \gamma_2^{b_2(T)}) \right| < \epsilon
\end{array}
\right\}
\end{equation*}
where $\cP$ is a finite ordered partition of $X$, $T \subseteq G$ is finite, $\epsilon > 0$, and $\gamma_2$ is a finite ordered partition of $Y$. We remark that this topology is not metrizable and not even Hausdorff.
\end{defn}

At times we will use an alternate system of neighborhood bases which give the same topology.

\begin{lem} \label{lem:iso}
Let $G \acts^a (X, \mu)$ be a {\pmp} action, and let $(b_2, \lambda_2) \in \sjoin(a)$. Let $\mathcal{A}$ be an algebra of Borel subsets of $Y$ with the property that the smallest $b_2$-invariant sub-$\sigma$-algebra containing $\mathcal{A}$ is the entire Borel $\sigma$-algebra $\Borel(Y)$ modulo $\pi^Y_*(\lambda_2)$-null sets. Then, for the topology defined in Definition \ref{defn:wcj}, a neighborhood base for $(b_2, \lambda_2)$ is given by the sets
\begin{equation*}
\left\{(b_1, \lambda_1) :
\begin{array}{l}
\exists \text{ ordered partition } \alpha_1 \text{ of } Y \text{ with}\\
\quad\left|\dist_{\lambda_1}(\cP \vee \alpha_1^{b_1(T)}) - \dist_{\lambda_2}(\cP \vee \alpha_2^{b_2(T)}) \right| < \epsilon
\end{array}
\right\}
\end{equation*}
where $\cP$ is a finite ordered partition of $X$, $T \subseteq G$ is finite, $\epsilon > 0$, and $\alpha_2 \subseteq \mathcal{A}$ is a finite ordered partition.
\end{lem}

\begin{proof}
Write $U(\cP, \gamma_2, T, \epsilon)$ for the open neighborhoods of $(b_2, \lambda_2)$ defined in Definition \ref{defn:wcj}. The sets defined in the statement of the lemma are simply the sets $U(\cP, \alpha_2, T, \epsilon)$ where $\alpha_2 \subseteq \mathcal{A}$. These are open neighborhoods of $(b_2, \lambda_2)$ by definition. So we only need to show that for every choice of $\cP$, $\gamma_2$, $T$, and $\epsilon$ there are $\cP'$, $\alpha_2$, $T'$, and $\epsilon'$ with $U(\cP', \alpha_2, T', \epsilon') \subseteq U(\cP, \gamma_2, T, \epsilon)$.

Fix a finite ordered partition $\cP$ of $X$, a finite ordered partition $\gamma_2$ of $Y$, a finite $T \subseteq G$, and $\epsilon > 0$. By our assumption on $\mathcal{A}$, there is a finite partition $\alpha_2 \subseteq \mathcal{A}$, a finite $F \subseteq G$, and a coarsening $\hat{\gamma}_2 \leq \alpha_2^{b_2(F)}$ satisfying $\dB_{\lambda_2}(\hat{\gamma}_2, \gamma_2) < \epsilon / (2 |T|)$. Now consider a joining $(b_1, \lambda_1) \in U(\cP, \alpha_2, F T, \epsilon / 2)$. Then there is a partition $\alpha_1$ of $Y$ satisfying
$$|\dist_{\lambda_1}(\cP \vee \alpha_1^{b_1(F T)}) - \dist_{\lambda_2}(\cP \vee \alpha_2^{b_2(F T)})| < \epsilon / 2.$$
Let $\gamma_1 \leq \alpha_1^{b_1(F)}$ be built from $\alpha_1^{b_1(F)}$ as $\hat{\gamma}_2$ is built from $\alpha_2^{b_2(F)}$. Then $\gamma_1^{b_1(T)} \leq \alpha_1^{b_1(F T)}$ and $\hat{\gamma_2}^{b_2(T)} \leq \alpha_2^{b_2(F T)}$ and hence
$$|\dist_{\lambda_1}(\cP \vee \gamma_1^{b_1(T)}) - \dist_{\lambda_2}(\cP \vee \hat{\gamma}_2^{b_2(T)})| < \epsilon / 2.$$
By construction we also have
$$|\dist_{\lambda_2}(\cP \vee \hat{\gamma}_2^{b_2(T)}) - \dist_{\lambda_2}(\cP \vee \gamma_2^{b_2(T)})| \leq \dB_{\lambda_2}(\hat{\gamma}_2^{b_2(T)}, \gamma_2^{b_2(T)}) < |T| \cdot \dB_{\lambda_2}(\hat{\gamma}_2, \gamma_2) < \epsilon / 2.$$
Thus $(b_1, \lambda_1) \in U(\cP, \gamma_2, T, \epsilon)$. This shows that $U(\cP, \alpha_2, F T, \epsilon / 2) \subseteq U(\cP, \gamma_2, T, \epsilon)$.
\end{proof}

Our main interest in this topology on the space of joinings is the following notion of comparison which it provides.

\begin{defn}
Fix a {\pmp} action $G \acts^a (X, \mu)$. For $(b_1, \lambda_1), (b_2, \lambda_2) \in \sjoin(a)$, we say that $G \acts^{a \times b_1} (X \times Y, \lambda_1)$ \emph{weakly contains} $G \acts^{a \times b_2} (X \times Y, \lambda_2)$ \emph{as joinings with} $G \acts^a (X, \mu)$ or, more briefly, $(b_1, \lambda_1)$ \emph{weakly contains $(b_2, \lambda_2)$ as joinings with} $a$, if every open set containing $(b_2, \lambda_2)$ contains $(b_1, \lambda_1)$. Similarly, we say that $(b_1, \lambda_1)$ and $(b_2, \lambda_2)$ are \emph{weakly equivalent as joinings with $a$} if each weakly contains the other as joinings with $a$.
\end{defn}

In other words, $(b_1, \lambda_1)$ weakly contains $(b_2, \lambda_2)$ as joinings with $a$ if and only if for every $\epsilon > 0$, every finite $T \subseteq G$, every finite ordered partition $\cP$ of $X$, and every finite ordered partition $\gamma_2$ of $Y$, there is a finite ordered partition $\gamma_1$ of $Y$ such that
$$|\dist_{\lambda_1}(\cP \vee \gamma_1^{b_1(T)}) - \dist_{\lambda_2}(\cP \vee \gamma_2^{b_2(T)})| < \epsilon.$$

We say that two joinings $(b_1, \lambda_1), (b_2, \lambda_2) \in \sjoin(a)$ are \emph{isomorphic} if there is an isomorphism $\psi$ from $G \acts^{b_1} (Y, \pi^Y_*(\lambda_1))$ to $G \acts^{b_2} (Y, \pi^Y_*(\lambda_2))$ which extends to an isomorphism $\id \times \psi : (X \times Y, \lambda_1) \rightarrow (X \times Y, \lambda_2)$. It is immediate from the definitions that the topology on $\sjoin(a)$ and the notion of weak containment for joinings with $a$ are both invariant under isomorphisms. One could define more general notions of isomorphism for joinings, but we caution the reader that more general isomorphism notions might not respect the topology on $\sjoin(a)$ or the notion of weak containment of joinings. Since all uncountable standard Borel spaces are Borel isomorphic, there is no harm in replacing $Y$ with any uncountable standard Borel space we like. Frequently, when working with two joinings $(b_1, \lambda_1)$ and $(b_2, \lambda_2)$, we will find it notationally helpful to have two versions of $Y$: $Y_1$ for $(b_1, \lambda_1)$ and $Y_2$ for $(b_2, \lambda_2)$. When we do this, we will frequently omit indicating the action $b_i$ as it is implicitly associated with the space $Y_i$ being acted upon.

An easy consequence of the definitions is that if $G \acts (Y_1, \nu_1)$ weakly contains $G \acts (Y_2, \nu_2)$ and $G \acts (X, \mu)$ is any {\pmp} action, then $G \acts (X \times Y_1, \mu \times \nu_1)$ weakly contains $G \acts (X \times Y_2, \mu \times \nu_2)$ as joinings with $G \acts (X, \mu)$. The lemma below generalizes this fact.

\begin{lem} \label{lem:triprod}
Let $G \acts (X, \mu)$, $G \acts (Z, \eta)$, and $G \acts (Y_i, \nu_i)$, $i = 1, 2$, be {\pmp} actions. Assume that $G \acts (X, \mu)$, $G \acts (Y_1, \nu_1)$, and $G \acts (Y_2, \nu_2)$ factor onto $G \acts (Z, \eta)$ via maps $p$, $q_1$, and $q_2$ respectively. If $G \acts (Z \times Y_1, (q_1 \times \id)_*(\nu_1))$ weakly contains $G \acts (Z \times Y_2, (q_2 \times \id)_*(\nu_2))$ as joinings with $G \acts (Z, \eta)$, then $G \acts (X \times Y_1, \mu \relprod{\eta} \nu_1)$ weakly contains $G \acts (X \times Y_2, \mu \relprod{\eta} \nu_2)$ as joinings with $G \acts (X, \mu)$.
\end{lem}

\begin{proof}
Let $\mu = \int_Z \mu^z \ d \eta(z)$ and $\nu_i = \int_Z \nu_i^z \ d \eta(z)$ be the disintegrations of $\mu$ and $\nu_i$ over $\eta$. Note that $(q_i \times \id)_*(\nu_i) = \int_Z (q_i \times \id)_*(\nu_i^z) \ d \eta(z) = \int_Z \delta_z \times \nu_i^z \ d \eta(z)$. Therefore for $C \subseteq Y_i$ and $D \subseteq Z$ we have
$$\int_D \nu_i^z(C) \ d \eta(z) = \int_Z (\delta_z \times \nu_i^z)(D \cap C) \ d \eta(z) = (q_i \times \id)_*(\nu_i)(D \cap C).$$

Fix $\epsilon > 0$, finite $T \subseteq G$, and finite ordered partitions $\cP$ of $X$ and $\gamma_2 = \{C_1^2, \ldots, C_n^2\}$ of $Y_2$. By approximating the functions $z \mapsto \mu^z(P)$, $P \in \cP$, by step-functions, we can find a finite partition $\xi$ of $Z$ and real numbers $\{\mu^D(P) : D \in \xi, \ P \in \cP\}$ satisfying $\mu^D(P) \leq 1$ and $|\mu^z(P) - \mu^D(P)| < \epsilon / (2 \cdot |\cP| \cdot n^{|T|})$ for all $P \in \cP$, $D \in \xi$, and $z \in D$. By assumption, there is a partition $\gamma_1 = \{C_1^1, \ldots, C_n^1\}$ of $Y_1$ satisfying
$$|\dist_{(q_1 \times \id)_*(\nu_1)}(\xi \vee \gamma_1^T) - \dist_{(q_2 \times \id)_*(\nu_2)}(\xi \vee \gamma_2^T)| < \epsilon / (2 \cdot |\cP| \cdot n^{|T|}).$$
We index the sets in $\gamma_1^T$ and $\gamma_2^T$ by functions $f \in \{1, \ldots, n\}^T$ as follows. For $i = 1, 2$ and $f \in \{1, \ldots, n\}^T$ we set $C_f^i = \bigcap_{t \in T} t^{-1} \cdot C_{f(t)}^i$. For $P \in \cP$ and $f \in \{1, \ldots, n\}^T$ we have
\begin{align*}
& |\mu \relprod{\eta} \nu_1(P \cap C_f^1) - \mu \relprod{\eta} \nu_2(P \cap C_f^2)|\\
 & = \left| \int_Z \mu^z(P) \cdot \nu_1^z(C_f^1) - \mu^z(P) \cdot \nu_2^z(C_f^2) \ d \eta(z) \right|\\
 & < \frac{\epsilon}{2 \cdot |\cP| \cdot n^{|T|}} + \left| \sum_{D \in \xi} \int_D \mu^D(P) \cdot \Big( \nu_1^z(C_f^1) - \nu_2^z(C_f^2) \Big) \ d \eta(z) \right|\\
 & = \frac{\epsilon}{2 \cdot |\cP| \cdot n^{|T|}} + \left| \sum_{D \in \xi} \mu^D(P) \cdot \Big( (q_1 \times \id)_*(\nu_1)(D \cap C_f^1) - (q_2 \times \id)_*(\nu_2)(D \cap C_f^2) \Big) \right|\\
 & < \frac{\epsilon}{|\cP| \cdot n^{|T|}}.
\end{align*}
By summing over all $P \in \cP$ and $f \in \{1, \ldots, n\}^T$ we conclude
\begin{equation*}
|\dist_{\mu \relprod{\eta} \nu_1}(\cP \vee \gamma_1^T) - \dist_{\mu \relprod{\eta} \nu_2}(\cP \vee \gamma_2^T)| < \epsilon.\qedhere
\end{equation*}
\end{proof}

\section{Stabilizers and invariant random subgroups} \label{sec:stab}

For a countable group $G$ we let $\Sub(G)$ denote the space of all subgroups of $G$. A base for the topology on $\Sub(G)$ is given by the basic open sets $\{H \in \Sub(G) : H \cap T = F\}$ as $F \subseteq T$ range over the finite subsets of $G$. An \emph{invariant random subgroup}, or \emph{IRS}, of $G$ is a Borel probability measure $\theta$ on $\Sub(G)$ which is invariant under the conjugation action of $G$. This concept was first introduced in \cite{AGV}. Every {\pmp} action $G \acts (X, \mu)$ produces an IRS $\Stab_*(\mu)$ via the push-forward of $\mu$ under the stabilizer map $\Stab : X \rightarrow \Sub(G)$. We call $\Stab_*(\mu)$ the \emph{stabilizer type} of $G \acts (X, \mu)$.

Tucker-Drob proved that if two actions are weakly equivalent, then they must have the same stabilizer type \cite{TD12}. The main lemma of this section is a technical elaboration on this fact, showing that one can witness the weak containment while approximately preserving the stabilizer map. We will need the following simple notion and lemma.

\begin{defn}
Let $X$ be a set, $S \subseteq X$, and $\beta$ a partition of $X$. We say that $\beta$ \emph{separates} $S$ if every class of $\beta$ contains at most one element of $S$.
\end{defn}

\begin{lem} \label{lem:part}
Let $G \acts (X, \mu)$ be a {\pmp} action and let $T \subseteq G$ be finite. Then there is a finite Borel partition $\beta$ of $X$ such that $\beta$ separates $T \cdot x$ for every $x \in X$.
\end{lem}

\begin{proof}
Let $\Gamma$ be the Borel graph on $X$ defined by $(x, y) \in \Gamma \Leftrightarrow (x \neq y) \wedge (x \in T T^{-1} \cdot y)$. The degree of every $x \in X$ is bounded by $|T|^2 < \infty$, and thus by \cite[Prop. 4.6]{KST99} there is a finite Borel partition $\beta$ of $X$ such that $x$ and $y$ lie in different classes of $\beta$ whenever $(x, y) \in \Gamma$. Then, for every $x \in X$, $\beta$ separates $T \cdot x$ since there is an edge in $\Gamma$ between every pair of points in $T \cdot x$.
\end{proof}

We now present this section's main lemma.

\begin{lem} \label{lem:stab}
Let $G \acts (Y, \nu)$ and $G \acts (Z, \eta)$ be {\pmp} actions having the same stabilizer type $\theta$. Assume that $G \acts (Y, \nu)$ weakly contains $G \acts (Z, \eta)$. Then the factor joining $G \acts (\Sub(G) \times Y, (\Stab \times \id)_*(\nu))$ weakly contains $G \acts (\Sub(G) \times Z, (\Stab \times \id)_*(\eta))$ as joinings with $G \acts (\Sub(G), \theta)$.
\end{lem}

\begin{proof}
Fix $\epsilon > 0$, a finite $T \subseteq G$, a finite partition $\cP$ of $\Sub(G)$, and a finite partition $\gamma_2$ of $Z$. We must find a partition $\gamma_1$ of $Y$ with
$$|\dist_{(\Stab \times \id)_*(\nu)} (\cP \vee \gamma_1^T) - \dist_{(\Stab \times \id)_*(\eta)} (\cP \vee \gamma_2^T)| < \epsilon.$$
Since the sets $\{H \in \Sub(G) : H \cap W = K\}$, where $K \subseteq W \subseteq G$ are finite, form a base for the topology on $\Sub(G)$, there is a finite $W \subseteq G$ and a partition $\cP'$ of $\Sub(G)$ which is measurable with respect to the map $H \mapsto H \cap W$ and satisfies $\dB_\theta(\cP', \cP) < \epsilon / 2$. Let $\cQ = \{Q_K : K \subseteq W\}$ be the partition of $\Sub(G)$ defined by setting $Q_K = \{H \in \Sub(G) : H \cap W = K\}$. Since $\cP' \leq \cQ$, it suffices to find a partition $\gamma_1$ of $Y$ with
$$|\dist_{(\Stab \times \id)_*(\nu)} (\cQ \vee \gamma_1^T) - \dist_{(\Stab \times \id)_*(\eta)} (\cQ \vee \gamma_2^T)| < \epsilon / 2.$$
Denote by $Q_K^+$ the set $\{H \in \Sub(G) : K \subseteq H \cap W\}$.

Fix a finite partition $\beta$ of $Z$ which separates points in $W \cdot z$ for all $z \in Z$. Then the map $z \mapsto \Stab(z) \cap W$ is $\beta^W$-measurable. Let $\hat{\beta}$, $B_K$, and $B_K^+$ be the pre-images under the stabilizer map $\Stab : Z \rightarrow \Sub(G)$ of $\cQ$, $Q_K$, and $Q_K^+$, respectively. Note that $\hat{\beta} \leq \beta^W$ and
\begin{equation} \label{eqn:stab1}
\dist_{(\Stab \times \id)_*(\eta)}(\cQ \vee \gamma_2^T) = \dist_\eta(\hat{\beta} \vee \gamma_2^T).
\end{equation}

Since $G \acts (Y, \nu)$ weakly contains $G \acts (Z, \eta)$, there are partitions $\alpha, \gamma_1$ of $Y$ with
$$|\dist_\nu(\alpha^W \vee \gamma_1^T) - \dist_\eta(\beta^W \vee \gamma_2^T)| < 2^{-2|W|-2} \cdot \epsilon.$$
For $K \subseteq W$ let $A_K$ be the set of $y \in Y$ with the property that for every $w \in W$, $w \cdot y$ lies in the same piece of $\alpha$ as $y$ if and only if $w \in K$. Also let $A_K^+$ be the set of $y \in Y$ with the property that $K \cdot y$ is contained in a single class of $\alpha$. Let $\hat{\alpha} = \{A_K : K \subseteq W\}$. It is important to note that $\hat{\alpha}$ coarsens $\alpha^W$ in the same manner $\hat{\beta}$ coarsens $\beta^W$, and thus
\begin{equation} \label{eqn:stab2}
|\dist_\nu(\hat{\alpha} \vee \gamma_1^T) - \dist_\eta(\hat{\beta} \vee \gamma_2^T)| < \epsilon / 4.
\end{equation}
Since $y \in A_K^+$ whenever $\Stab(y) \in Q_K^+$, we see that for every $K \subseteq W$
$$\theta(Q_K^+) \leq \nu(A_K^+) <  \eta(B_K^+) + 2^{-2|W|-2} \cdot \epsilon = \theta(Q_K^+) + 2^{-2|W|-2} \cdot \epsilon.$$
Therefore $(\Stab \times \id)_*(\nu)(Q_K^+ \symd A_K^+) < 2^{-2|W|-2} \cdot \epsilon$, and
$$(\Stab \times \id)_*(\nu)(Q_K \symd A_K) \leq \sum_{K \subseteq U \subseteq W} (\Stab \times \id)_*(\nu)(Q_U^+ \symd A_U^+) < 2^{-|W|-2} \cdot \epsilon.$$
Summing over $K \subseteq W$ we obtain $\dB_{(\Stab \times \id)_*(\nu)}(\cQ, \hat{\alpha}) < \epsilon / 4$. The proof is now completed by combining (\ref{eqn:stab1}) and (\ref{eqn:stab2}) with the inequality
\begin{equation*}
|\dist_{(\Stab \times \id)_*(\nu)}(\cQ \vee \gamma_1^T) - \dist_{\nu}(\hat{\alpha} \vee \gamma_1^T)| < \epsilon / 4.\qedhere
\end{equation*}
\end{proof}

\begin{cor} \label{cor:wcstimes}
Let $G \acts (Y, \nu)$ and $G \acts (Z, \eta)$ be {\pmp} actions having the same stabilizer type $\theta$, and assume that $G \acts (Y, \nu)$ weakly contains $G \acts (Z, \eta)$. Then for any {\pmp} action $G \acts (X, \mu)$ with stabilizer type $\theta$, we have $G \acts (X \times Y, \mu \relprod{\theta} \nu)$ weakly contains $G \acts (X \times Z, \mu \relprod{\theta} \eta)$ as joinings with $G \acts (X, \mu)$.
\end{cor}

\begin{proof}
Combine Lemmas \ref{lem:stab} and \ref{lem:triprod}.
\end{proof}

\section{Non-free Bernoulli shifts} \label{sec:bs}

Every IRS of $G$ is the stabilizer type of some {\pmp} action of $G$ \cite{AGV}. This fact follows from the construction of non-free Bernoulli shifts which we now discuss. Let $(L, \lambda)$ be a standard probability space with $\lambda$ not a single-point mass. We let $G$ act on $L^G$ by the standard left-shift action: $(g \cdot x)(t) = x(g^{-1} t)$ for $g, t \in G$ and $x \in L^G$. For $H \in \Sub(G)$, we identify $L^{H \backslash G}$ with the set of points $x \in L^G$ with $H \subseteq \Stab(x)$, and we consider the corresponding Borel probability measure $\lambda^{H \backslash G}$ on $L^G$ which is supported on $L^{H \backslash G}$. If $\theta$ is an IRS of $G$ which is supported on the infinite-index subgroups of $G$, then we define the \emph{non-free Bernoulli shift with stabilizer type $\theta$ and with base space $(L, \lambda)$} to be the standard shift-action of $G$ on $L^G$ equipped with the $G$-invariant probability measure
$$\lambda^{\theta \backslash G} : = \int_{H \in \Sub(G)} \lambda^{H \backslash G} \ d \theta(H).$$
If $H \in \Sub(G)$ has infinite index in $G$, then $\Stab(x) = H$ for $\lambda^{H \backslash G}$-almost-every $x \in L^G$. Thus $\theta$ is indeed the stabilizer type of $G \acts (L^G, \lambda^{\theta \backslash G})$. Note that if $\theta = \Stab_*(\mu)$ for a {\pmp} action $G \acts (X, \mu)$, then $\theta$ is supported on the infinite-index subgroups of $G$ if and only if the action $G \acts (X, \mu)$ is aperiodic.

It was proven by Ab\'{e}rt--Weiss that every free {\pmp} action of $G$ weakly contains all (free) Bernoulli shifts over $G$ \cite{AW13}. This was extended by Tucker-Drob, who proved that every aperiodic {\pmp} action $G \acts (X, \mu)$ of stabilizer type $\theta$ is weakly equivalent to $G \acts (X \times L^G, \mu \relprod{\theta} \lambda^{\theta \backslash G})$, and in particular weakly contains $G \acts (L^G, \lambda^{\theta \backslash G})$ \cite{TD12}. We reconstruct the proofs of Ab\'{e}rt--Weiss and Tucker-Drob in our context of weak containment of joinings. Our main interest, however, is the class of actions which both have stabilizer type $\theta$ and are weakly contained in all {\pmp} actions of stabilizer type $\theta$. This includes, by the result of Tucker-Drob, the non-free Bernoulli shifts $(L^G, \lambda^{\theta \backslash G})$.

\begin{lem} \label{lem:awcent}
Let $G \acts (X, \mu)$ and $G \acts (Z, \eta)$ be {\pmp} actions with $G \acts (Z, \eta)$ aperiodic, and let $\lambda$ be a joining of these two actions. Let $\theta = \Stab \circ \pi^Z_*(\lambda)$ be the stabilizer type of $G \acts (Z, \eta)$. Assume that $G \acts (Y, \nu)$ has stabilizer type $\theta$ and is weakly contained in all {\pmp} actions of stabilizer type $\theta$. Then the joinings $G \acts (X \times Z, \lambda)$ and $G \acts (X \times Z \times Y, \lambda \relprod{\theta} \nu)$ are weakly equivalent as joinings with $G \acts (X, \mu)$.
\end{lem}

\begin{proof}
It is immediate from the definitions that $G \acts (X \times Z \times Y, \lambda \relprod{\theta} \nu)$ weakly contains $G \acts (X \times Z, \lambda)$ as joinings with $G \acts (X, \mu)$. So it suffices to show the reverse weak containment. By our assumption on $G \acts (Y, \nu)$ and Lemmas \ref{lem:triprod} and \ref{lem:stab} we know that $G \acts (X \times Z \times 2^G, \lambda \relprod{\theta} u_2^{\theta \backslash G})$ weakly contains $G \acts (X \times Z \times Y, \lambda \relprod{\theta} \nu)$ as joinings with $G \acts (X, \mu)$. Therefore it suffices to show that $G \acts (X \times Z, \lambda)$ weakly contains $G \acts (X \times Z \times 2^G, \lambda \relprod{\theta} u_2^{\theta \backslash G})$ as joinings with $G \acts (X, \mu)$.

Let $\lambda = \int_{\Sub(G)} \lambda_H \ d \theta(H)$ be the disintegration of $\lambda$ with respect to the map $\Stab \circ \pi^Z : X \times Z \rightarrow \Sub(G)$. We note that $\lambda \relprod{\theta} u_2^{\theta \backslash G} = \int_{\Sub(G)} \lambda_H \cdot u_2^{H \backslash G} \ d \theta(H)$. Throughout the proof we will implicitly identify $\Borel(X) \subseteq \Borel(X \times Z)$ and $\Borel(Z) \subseteq \Borel(X \times Z)$ in the natural way. Let $\xi = \{D_0, D_1\}$ be the canonical generating partition for $2^G$, where
$$D_i = \{y \in 2^G : y(1_G) = i\}.$$
Fix a finite partition $\cP$ of $X$, a finite partition $\cQ$ of $Z$, a finite set $T \subseteq G$, and $0 < \epsilon < 1$. We will build a partition $\gamma = \{C_0, C_1\}$ of $Z$ such that
$$\left| \dist_\lambda(\cP \vee \cQ \vee \gamma^T) - \dist_{\lambda \relprod{\theta} u_2^{\theta \backslash G}}(\cP \vee \cQ \vee \xi^T) \right| < \epsilon.$$
Lemma \ref{lem:iso} implies that this will be sufficient to prove this lemma.

Fix $0 < \delta < 2^{-3 |T|} \cdot |\cP|^{-3} \cdot |\cQ|^{-3} \cdot \epsilon^3$. Apply Lemma \ref{lem:part} to $Z$ to obtain a finite partition $\beta$ of $Z$ such that $\beta$ separates $T \cdot z$ for every $z \in Z$. Since any partition finer than $\beta$ has this same property and since $(Z, \eta)$ has no atoms (by aperiodicity), we may make $\beta$ finer if necessary so that
$$\sum_{B \in \beta} |T|^2 \cdot \eta(B)^2 < \delta.$$
We will implicitly also view $\beta$ as a partition of $X \times Z$.

Say $\beta = \{B_0, \ldots, B_{k-1}\}$. Let $\omega \in \{0, 1\}^k$ be a random variable with law $u_2^k$. Define a random partition $\gamma(\omega) = \{C_0(\omega), C_1(\omega)\} \subseteq \Borel(Z)$ of $X \times Z$ by setting
$$C_i(\omega) = \bigcup \{B_m : 0 \leq m < k \ \omega(m) = i\}.$$
We will check that with high probability the random partition $\gamma(\omega)$ has the desired property.

We index the sets in $\gamma(\omega)^T$ by the functions $f \in \{0, 1\}^T$, where
$$C_f(\omega) = \bigcap_{t \in T} t^{-1} \cdot C_{f(t)}(\omega).$$
We similarly define
$$D_f = \bigcap_{t \in T} t^{-1} \cdot D_{f(t)} \in \xi^T.$$
Fix $f \in \{0, 1\}^T$ and let $1_{C_f(\omega)}$ denote the characteristic function of $C_f(\omega)$. Since for $z \in  Z$ the partition $\beta$ separates $T \cdot z$, we see that for $(x, z) \in X \times Z$ the quantity $1_{C_f(\omega)}(x, z)$ has expected value
$$\Eavg_\omega 1_{C_f(\omega)}(x, z) = u_2^{\Stab(z) \backslash G}(D_f).$$
For $P \in \cP$ and $Q \in \cQ$ we can integrate $1_{C_f(\omega)}(x, z)$ over $(x, z) \in P \cap Q$ and use Fubini's theorem to obtain
\begin{align*}
\Eavg_\omega \lambda(P \cap Q \cap C_f(\omega)) & = \int_{P \cap Q} u_2^{\Stab(z) \backslash G}(D_f) \ d \lambda(x, z)\\
 & = \int_{\Sub(G)} \lambda_H(P \cap Q) \cdot u_2^{H \backslash G}(D_f) \ d \theta(H)\\
 & = \lambda \relprod{\theta} u_2^{\theta \backslash G}(P \cap Q \cap D_f).
\end{align*}

Now we estimate the variance of $\lambda(P \cap Q \cap C_f(\omega))$. Set
$$\Delta = \bigcup_{B \in \beta} (T^{-1} \cdot B) \times (T^{-1} \cdot B) \subseteq (X \times Z) \times (X \times Z).$$
Note that
$$\lambda \times \lambda (\Delta) \leq \sum_{B \in \beta} |T|^2 \cdot \eta(B)^2 < \delta.$$
Also observe that if $((x, z), (x', z')) \not\in \Delta$ then $\beta$ separates $T \cdot z \cup T \cdot z'$. Therefore, for $((x, z), (x', z')) \not\in \Delta$ the product $1_{C_f(\omega)}(x, z) \cdot 1_{C_f(\omega)}(x', z')$ has expected value $u_2^{\Stab(z) \backslash G}(D_f) \cdot u_2^{\Stab(z') \backslash G}(D_f)$. So we have
\begin{align*}
\Eavg_\omega & \lambda(P \cap Q \cap C_f(\omega))^2\\
 & = \int_{(P \cap Q) \times (P \cap Q)} \Eavg_\omega 1_{C_f(\omega)}(x, z) \cdot 1_{C_f(\omega)}(x', z') \ d (\lambda \times \lambda)((x, z), (x', z))\\
 & < \delta + \int_{(P \cap Q) \times (P \cap Q)} u_2^{\Stab(z) \backslash G}(D_f) \cdot u_2^{\Stab(z') \backslash G}(D_f) \ d (\lambda \times \lambda)((x, z), (x', z'))\\
& = \delta + \left( \int_{\Sub(G)} \lambda_H(P \cap Q) \cdot u_2^{H \backslash G}(D_f) \ d \theta(H) \right)^2\\
& = \delta + \lambda \relprod{\theta} u_2^{\theta \backslash G}(P \cap Q \cap D_f)^2.
\end{align*}
Therefore the variance is
$$\Var_\omega \lambda(P \cap Q \cap C_f(\omega)) = \Eavg_\omega \lambda(P \cap Q \cap C_f(\omega))^2 - (\Eavg_\omega \lambda(P \cap Q \cap C_f(\omega)))^2 < \delta.$$
The Chebyshev inequality implies that for every $r > 0$
$$u_2^k \Big( \Big\{ \omega : |\lambda(P \cap Q \cap C_f(\omega)) - \Eavg_\omega \lambda(P \cap Q \cap C_f(\omega))| > r \Big\} \Big) \leq \frac{\Var_\omega \lambda(P \cap Q \cap C_f(\omega))}{r^2}.$$
Using $r = \delta^{1/3}$ we obtain
$$u_2^k \Big( \Big\{ \omega : |\lambda(P \cap Q \cap C_f(\omega)) - \Eavg_\omega \lambda(P \cap Q \cap C_f(\omega))| > \delta^{1/3} \Big\} \Big) \leq \delta^{1/3}.$$
Since $2^{|T|} \cdot |\cP| \cdot |\cQ| \cdot \delta^{1/3} < \epsilon < 1$, it follows that there is $\omega \in \{0, 1\}^k$ such that
\begin{equation*}
\left| \dist_\lambda(\cP \vee \cQ \vee \gamma(\omega)^T) - \dist_{\lambda \relprod{\theta} u_2^{\theta \backslash G}}(\cP \vee \cQ \vee \xi^T) \right| < \epsilon.\qedhere
\end{equation*}
\end{proof}

Given two {\pmp} actions $G \acts (X, \mu)$ and $G \acts (Y, \nu)$ having the same stabilizer type, and given a joining $\lambda$ of $\mu$ with $\nu$, we say that $\lambda$ \emph{preserves stabilizers} if
$$\lambda(\{(x, y) \in X \times Y : \Stab(x) = \Stab(y)\}) = 1.$$
When $G \acts (X, \mu)$ and $G \acts (Y, \nu)$ have the same stabilizer type $\theta$, there always exists at least one joining which preserves stabilizers, namely the relatively independent joining $\mu \relprod{\theta} \nu$.

\begin{cor} \label{cor:awstab}
Let $G \acts (X, \mu)$ and $G \acts (Z, \eta)$ be aperiodic {\pmp} actions having the same stabilizer type $\theta$, and let $\lambda$ be a stabilizer-preserving joining of these two actions. Assume that $G \acts (Y, \nu)$ has stabilizer type $\theta$ and is weakly contained in all other {\pmp} actions of stabilizer type $\theta$. Then $G \acts (X \times Z, \lambda)$ weakly contains $G \acts (X \times Y, \mu \relprod{\theta} \nu)$ as joinings with $G \acts (X, \mu)$.
\end{cor}

\begin{proof}
Let $\mu = \int \mu_H \ d \theta(H)$, $\nu = \int \nu_H \ d \theta(H)$, and $\lambda = \int \lambda_H \ d \theta(H)$ be the disintegrations of $\mu$, $\nu$, and $\lambda$ over $\theta$. Since $\lambda$ preserves stabilizers, we have that $\pi^X_*(\lambda_H)$ is supported on points in $X$ having stabilizer $H$. Since also $\mu = \int \pi^X_*(\lambda_H) \ d \theta(H)$, uniqueness of disintegrations implies that $\pi^X_*(\lambda_H) = \mu_H$ for $\theta$-almost-every $H$. It follows that
$$(\pi^X \times \pi^Y)_*(\lambda \relprod{\theta} \nu) = \int (\pi^X \times \pi^Y)_*(\lambda_H \times \nu_H) \ d \theta(H) = \int \mu_H \times \nu_H \ d \theta(H) = \mu \relprod{\theta} \nu.$$
Therefore $G \acts (X \times Z \times Y, \lambda \relprod{\theta} \nu)$ factors onto $G \acts (X \times Y, \mu \relprod{\theta} \nu)$ via the map $\pi^X \times \pi^Y$. In particular, $G \acts (X \times Z \times Y, \lambda \relprod{\theta} \nu)$ weakly contains $G \acts (X \times Y, \mu \relprod{\theta} \nu)$ as joinings with $G \acts (X, \mu)$. By the previous lemma $G \acts (X \times Z, \lambda)$ weakly contains $G \acts (X \times Z \times Y, \lambda \relprod{\theta} \nu)$ as joinings with $G \acts (X, \mu)$. Now note that weak containment is a transitive property.
\end{proof}

\section{Relative Rokhlin entropy and joinings} \label{sec:ineq}

In this section we show that weak containment of joinings leads to inequalities in relative Rokhlin entropies.

Recall that a real-valued function $f$ on a topological space $X$ is called \emph{upper-semicontinuous} if for every $x \in X$ and $\epsilon > 0$ there is an open set $U$ containing $x$ with $f(y) < f(x) + \epsilon$ for all $y \in U$. When $X$ is first countable, this is equivalent to saying that $f(x) \geq \limsup f(x_n)$ whenever $(x_n)$ is a sequence converging to $x$. We observe a simple property.

\begin{lem} \label{lem:limups}
Let $X$ be a topological space, let $f_\epsilon : X \rightarrow [0, \infty)$, $\epsilon > 0$, be a family of upper-semicontinuous functions and set $g = \lim_{\epsilon \rightarrow 0} f_\epsilon$. Assume that
$$g(x) - \epsilon \leq f_\epsilon(x) \leq g(x)$$
for all $\epsilon > 0$ and all $x \in X$. Then $g : X \rightarrow \R$ is upper-semicontinuous.
\end{lem}

\begin{proof}
Fix $x \in X$ and $\epsilon > 0$. Since $f_{\epsilon/2}$ is upper-semicontinuous, there is an open neighborhood $U$ of $x$ with $f_{\epsilon/2}(y) < f_{\epsilon/2}(x) + \epsilon/2$ for all $y \in U$. Then for $y \in U$ we have $g(y) \leq f_{\epsilon/2}(y) + \epsilon/2 \leq f_{\epsilon/2}(x) + \epsilon \leq g(x) + \epsilon$.
\end{proof}

Fix an action $G \acts^a (X, \mu)$. Let $f : \sjoin(a) \rightarrow \R$ be an upper-semicontinuous function. Recall that the topology on $\sjoin(a)$ is such that $(b_1, \lambda_1)$ weakly contains $(b_2, \lambda_2)$ as joinings with $a$ if and only if every open neighborhood of $(b_2, \lambda_2)$ contains $(b_1, \lambda_1)$. Thus, if $(b_1, \lambda_1)$ weakly contains $(b_2, \lambda_2)$ as joinings with $a$ then $f(b_1, \lambda_1) \leq f(b_2, \lambda_2)$.

For an action $G \acts^a (X, \mu)$, when we wish to emphasize the action $a$ we write $\rh_a(X, \mu \given \cF)$ and $\rh_{a, \mu}(\mathcal{C} \given \cF)$ for the Rokhlin entropies $\rh_G(X, \mu \given \cF)$ and $\rh_{G,\mu}(\mathcal{C} \given \cF)$, respectively.

\begin{lem} \label{lem:upjoin}
Let $G$ be a countable group, let $G \acts^a (X, \mu)$ be an aperiodic {\pmp} action, and let $\cF$ be a $G$-invariant sub-$\sigma$-algebra. If $\cP$ is a countable partition of $X$ with $\sH(\cP \given \cF) < \infty$ then the maps
\begin{align*}
(b, \lambda) \in \sjoin(a) & \mapsto \rh_{a \times b, \lambda}(\cP \given \cF), \quad \text{and}\\
(b, \lambda) \in \sjoin(a) & \mapsto \rh_{a \times b, \lambda}(\cP \given \cF \vee \Borel(Y))
\end{align*}
are upper-semicontinuous.
\end{lem}

\begin{proof}
Fix a countable partition $\cP$ of $X$ with $\sH(\cP \given \cF) < \infty$. For each $\epsilon > 0$ fix a finite partition $\cP_\epsilon$ which is coarser than $\cP$ and satisfies $\sH(\cP \given \cP_\epsilon \vee \cF) < \epsilon$. By sub-additivity of Rokhlin entropy, for every $(b, \lambda) \in \sjoin(a)$ we have
$$\rh_{a \times b, \lambda}(\cP \given \cF) - \epsilon \leq \rh_{a \times b, \lambda}(\cP_\epsilon \given \cF) \leq \rh_{a \times b, \lambda}(\cP \given \cF).$$
So by Lemma \ref{lem:limups}, the map $(b, \lambda) \mapsto \rh_{a \times b, \lambda}(\cP \given \cF)$ is upper-semicontinuous provided that $(b, \lambda) \mapsto \rh_{a \times b, \lambda}(\cP_\epsilon \given \cF)$ is upper-semicontinuous for every $\epsilon$. The same is true for $\rh_{a \times b, \lambda}(\cP \given \cF \vee \Borel(Y))$. Thus, it suffices to consider the case where $\cP$ is finite.

Fix a finite partition $\cP$ of $X$. Fix finite labeled partitions $\alpha \subseteq \Borel(X)$ and $\xi \subseteq \cF$, fix a finite $T \subseteq G$, and fix $\epsilon > 0$. For a finite partition $\gamma \subseteq \Borel(Y)$ define $f_{\epsilon,T}^{\alpha,\xi}(b, \lambda, \gamma)$ to be
$$\inf \Big\{ \sH_\lambda(\beta \given \chi^{a \times b(T)} \vee \xi) : \beta, \chi \leq \alpha \vee \gamma, \ \sH_\lambda(\chi) + \sH_\lambda(\cP \given (\beta \vee \chi)^{a \times b(T)} \vee \xi) < \epsilon \Big\}.$$
Since $G \acts (X, \mu)$ is aperiodic, \cite[Lem. 6.3]{AS} states that for every $\epsilon > 0$
$$\rh_{a \times b, \lambda}(\cP \given \cF) - \epsilon \leq \inf_{\alpha, \xi, T} \inf_{\gamma} f_{\epsilon,T}^{\alpha,\xi}(b, \lambda, \gamma) \leq \rh_{a \times b, \lambda}(\cP \given \cF).$$
Since $\beta, \chi \leq \alpha \vee \gamma$, we have that
$$\cP \vee \xi \vee \beta^{a \times b(T)} \vee \chi^{a \times b(T)} \leq \cP \vee \xi \vee \alpha^{a(T)} \vee \gamma^{b(T)}.$$
Since there are only finitely many choices for $\beta$ and $\chi$, $f_{\epsilon,T}^{\alpha,\xi}(b, \lambda, \gamma)$ is an upper-semicontinuous function of the (labeled) distribution
$$\dist_\lambda(\cP \vee \xi \vee \alpha^{a(T)} \vee \gamma^{b(T)}).$$
Since $\cP$, $\alpha$, and $\xi$ are partitions of $X$, from the definition of the topology on $\sjoin(a)$ it follows that
$$f_{\epsilon,T}^{\alpha,\xi}(b, \lambda) = \inf_{\gamma \subseteq \Borel(Y)} f_{\epsilon,T}^{\alpha,\xi}(b, \lambda, \gamma)$$
is an upper-semicontinuous function of $(b, \lambda)$. Taking more infimums does not destroy upper-semicontinuity, so
$$f_\epsilon(b, \lambda) = \inf_{\alpha,\xi,T} f_{\epsilon,T}^{\alpha,\xi}(b, \lambda)$$
is upper-semicontinuous. By Lemma \ref{lem:limups} it follows that
$$\lim_{\epsilon \rightarrow 0} f_\epsilon(b, \lambda) = \rh_{a \times b, \lambda}(\cP \given \cF)$$
is an upper-semicontinuous function of $(b, \lambda)$.

The proof for $\rh_{a \times b, \lambda}(\cP \given \cF \vee \Borel(Y))$ is nearly identical and merely involves defining $f_{\epsilon,T}^{\alpha,\xi}(b, \lambda, \gamma)$ to be
\begin{equation*}
\inf \Big\{ \sH_\lambda(\beta \given \chi^{a \times b(T)} \vee \xi \vee \gamma) : \beta, \chi \leq \alpha \vee \gamma, \ \sH_\lambda(\chi) + \sH_\lambda(\cP \given (\beta \vee \chi)^{a \times b(T)} \vee \xi \vee \gamma) < \epsilon \Big\}.\qedhere
\end{equation*}
\end{proof}

\begin{cor} \label{cor:upjoin}
Let $G$ be a countable group, let $G \acts^a (X, \mu)$ be an aperiodic {\pmp} action, and let $\cF$ be a $G$-invariant sub-$\sigma$-algebra. If $\rh_G(X, \mu \given \cF) < \infty$ then the maps
\begin{align*}
(b, \lambda) \in \sjoin(a) & \mapsto \rh_{a \times b, \lambda}(\Borel(X) \given \cF), \quad \text{and}\\
(b, \lambda) \in \sjoin(a) & \mapsto \rh_{a \times b}(X \times Y, \lambda \given \cF \vee \Borel(Y))
\end{align*}
are upper-semicontinuous.
\end{cor}

\begin{proof}
Let $\sinv_G^X$ denote the $\sigma$-algebra of $G$-invariant Borel subsets of $X$. Since $\rh_G(X, \mu \given \cF) < \infty$, there is a partition $\cP$ of $X$ with $\sH(\cP \given \cF \vee \sinv_G^X) < \infty$ and $\Borel(X) = \salg_G(\cP) \vee \cF \vee \sinv_G^X$. Now apply Lemma \ref{lem:upjoin} and use the facts that
\begin{align*}
\rh_{a \times b, \lambda}(\Borel(X) \given \cF) & = \rh_{a \times b, \lambda}(\cP \given \cF \vee \sinv_G^X), \quad \text{and}\\
\rh_{a \times b}(X \times Y, \lambda \given \cF \vee \Borel(Y)) & = \rh_{a \times b, \lambda}(\cP \given \cF \vee \sinv_G^X \vee \Borel(Y)).\qedhere
\end{align*}
\end{proof}

We do not know if $(b, \lambda) \in \sjoin(a) \mapsto \rh_{a \times b}(X \times Y, \lambda \given \cF \vee \Borel(Y))$ is an upper-semicontinuous function in general. However, weak containment of joinings always produces an inequality in relative Rokhlin entropies.

\begin{cor} \label{cor:wcjoin}
Let $G$ be a countable group, let $G \acts^a (X, \mu)$ be an aperiodic {\pmp} action, and let $\cF$ be a $G$-invariant sub-$\sigma$-algebra. Let $(b_1, \lambda_1), (b_2, \lambda_2) \in \sjoin(a)$ and assume that $(b_1, \lambda_1)$ weakly contains $(b_2, \lambda_2)$.
\begin{enumerate}
\item[\rm (1)] If $\cP$ is a countable partition of $X$ with $\sH(\cP \given \cF) < \infty$ then
$$\rh_{a \times b_1, \lambda_1}(\cP \given \cF) \leq \rh_{a \times b_2, \lambda_2}(\cP \given \cF).$$
\item[\rm (2)] If $\cP$ is a countable partition of $X$ with $\sH(\cP \given \cF) < \infty$ then
$$\rh_{a \times b_1, \lambda_1}(\cP \given \cF \vee \Borel(Y_1)) \leq \rh_{a \times b_2, \lambda_2}(\cP \given \cF \vee \Borel(Y_2)).$$
\item[\rm (3)] If $\rh_G(X, \mu \given \cF) < \infty$ then
$$\rh_{a \times b_1, \lambda_1}(\Borel(X) \given \cF) \leq \rh_{a \times b_2, \lambda_2}(\Borel(X) \given \cF).$$
\item[\rm (4)] Without any additional assumptions we have
$$\rh_{a \times b_1}(X \times Y_1, \lambda_1 \given \cF \vee \Borel(Y_1)) \leq \rh_{a \times b_2}(X \times Y_2, \lambda_2 \given \cF \vee \Borel(Y_2)).$$
\end{enumerate}
\end{cor}

\begin{proof}
Items (1) and (2) are immediate consequences of Lemma \ref{lem:upjoin}, and (3) is an immediate consequence of Corollary \ref{cor:upjoin}. Item (4) also follows from Corollary \ref{cor:upjoin} when $\rh_G(X, \mu \given \cF) < \infty$. So we must prove (4) in general.

Let $(\alpha_n)_{n \in \N}$ be an increasing sequence of finite partitions satisfying $\Borel(X) = \bigvee_{n \in \N} \salg_{a(G)}(\alpha_n)$. For each $n \in \N$ and $i = 1, 2$ let $G \acts (Z_{n,i}, \eta_{n,i})$ be the factor of $G \acts^{a \times b_i} (X \times Y_i, \lambda_i)$ associated to $\salg_{a(G)}(\alpha_n) \vee \cF \vee \Borel(Y_i)$. By clause (2), for all $n \leq m \in \N$ we have
\begin{align*}
 & \rh_{a \times b_1, \lambda_1}(\Borel(Z_{m,1}) \given \Borel(Z_{n,1}) \vee \cF \vee \Borel(Y_1))\\
 & = \rh_{a \times b_1, \lambda_1}(\alpha_m \given \salg_{a(G)}(\alpha_n) \vee \cF \vee \Borel(Y_1))\\
 & \leq \rh_{a \times b_2, \lambda_2}(\alpha_m \given \salg_{a(G)}(\alpha_n) \vee \cF \vee \Borel(Y_2))\\
 & = \rh_{a \times b_2, \lambda_2}(\Borel(Z_{m,2}) \given \Borel(Z_{n,2}) \vee \cF \vee \Borel(Y_2))
\end{align*}
and similarly
$$\rh_{a \times b_1, \lambda_1}(\Borel(Z_{m,1}) \given \cF \vee \Borel(Y_1)) \leq \rh_{a \times b_2, \lambda_2}(\Borel(Z_{m,2}) \given \cF \vee \Borel(Y_2)).$$
So by Theorem \ref{thm:ks} we conclude that
\begin{equation*}
\rh_{a \times b_1}(X \times Y_1, \lambda_1 \given \cF \vee \Borel(Y_1)) \leq \rh_{a \times b_2}(X \times Y_2, \lambda_2 \given \cF \vee \Borel(Y_2)).\qedhere
\end{equation*}
\end{proof}

We point out that clause (4) of the previous corollary is Lemma \ref{intro:lem} from the introduction.

For free actions $G \acts (X, \mu)$ and $G \acts (Y, \nu)$ of an amenable group $G$, it is well known that $h_G(X \times Y, \mu \times \nu \given \Borel(Y)) = h_G(X, \mu)$. It is an interesting question to ask if the same is true for non-amenable groups $G$. Below we answer this question positively under a weak containment assumption on $G \acts (Y, \nu)$. We also allow for non-free actions.

\begin{thm} \label{thm:prod}
Let $G \acts (X, \mu)$ be an aperiodic {\pmp} action of stabilizer type $\theta$, and let $\cF$ be a $G$-invariant sub-$\sigma$-algebra. Let $G \acts (Y, \nu)$ have stabilizer type $\theta$ and be weakly contained in all {\pmp} actions of stabilizer type $\theta$.
\begin{enumerate}
\item[\rm (1)] If $\cP$ is a countable partition of $X$ with $\sH(\cP \given \cF) < \infty$ then
$$\rh_{G,\mu}(\cP \given \cF) = \rh_{G, \mu \relprod{\theta} \nu}(\cP \given \cF) = \rh_{G, \mu \relprod{\theta} \nu}(\cP \given \cF \vee \Borel(Y)).$$
\item[\rm (2)] Without any additional assumptions we have
$$\rh_G(X, \mu \given \cF) = \rh_{G, \mu \relprod{\theta} \nu}(\Borel(X) \given \cF) = \rh_G(X \times Y, \mu \relprod{\theta} \nu \given \cF \vee \Borel(Y)).$$
\end{enumerate}
\end{thm}

\begin{proof}
Fix a partition $\cP$ of $X$ with $\sH(\cP \given \cF) < \infty$. It is immediate from the definitions that
$$\rh_{G,\mu}(\cP \given \cF) \geq \rh_{G, \mu \relprod{\theta} \nu}(\cP \given \cF) \geq \rh_{G, \mu \relprod{\theta} \nu}(\cP \given \cF \vee \Borel(Y)), \quad \text{and}$$
$$\rh_G(X, \mu \given \cF) \geq \rh_{G, \mu \relprod{\theta} \nu}(\Borel(X) \given \cF) \geq \rh_G(X \times Y, \mu \relprod{\theta} \nu \given \cF \vee \Borel(Y)).$$
So in both cases we only need to show that the right-most expression is greater than or equal to the left-most expression.

Fix $\epsilon > 0$. By Theorem \ref{thm:robin} there is a factor $G \acts (Z, \eta)$ of $(X, \mu)$, say via $f : (X, \mu) \rightarrow (Z, \eta)$, such that $\rh_G(Z, \eta) < \epsilon$ and $\Stab(f(x)) = \Stab(x)$ for every $x \in X$. The factor map $f$ from $(X, \mu)$ to $(Z, \eta)$ naturally produces a stabilizer-preserving joining $\rho = (\id \times f)_*(\mu)$. Of course, $G \acts (X \times Z, \rho)$ is isomorphic to $G \acts (X, \mu)$. By sub-additivity of Rokhlin entropy we have
\begin{align*}
\rh_G(X, \mu \given \cF) - \epsilon & < \rh_G(X, \mu \given \cF) - \rh_G(Z, \eta)\\
 & \leq \rh_G(X, \mu \given \cF \vee f^{-1}(\Borel(Z)))\\
 & = \rh_G(X \times Z, \rho \given \cF \vee \Borel(Z)).
\end{align*}
By the same reasoning, $\rh_{G,\mu}(\cP \given \cF) - \epsilon < \rh_{G, \rho}(\cP \given \cF \vee \Borel(Z))$. By Corollary \ref{cor:awstab} $G \acts (X \times Z, \rho)$ weakly contains $G \acts (X \times Y, \mu \relprod{\theta} \nu)$ as joinings with $G \acts (X, \mu)$. So Corollary \ref{cor:wcjoin}.(4) gives
$$\rh_G(X, \mu \given \cF) - \epsilon < \rh_G(X \times Z, \rho \given \cF \vee \Borel(Z)) \leq \rh_G(X \times Y, \mu \relprod{\theta} \nu \given \cF \vee \Borel(Y)),$$
and similarly Corollary \ref{cor:wcjoin}.(2) gives
$$\rh_{G,\mu}(\cP \given \cF) - \epsilon < \rh_{G, \rho}(\cP \given \cF \vee \Borel(Z)) \leq \rh_{G, \mu \relprod{\theta} \nu}(\cP \given \cF \vee \Borel(Y)).$$
Now let $\epsilon$ go to $0$.
\end{proof}

We restate the previous theorem in terms of free actions. This allows us to slightly relax our assumptions. Specifically, we assume that $G \acts (Y, \nu)$ is weakly contained in all free {\pmp} actions of $G$, but we do not assume that $G \acts (Y, \nu)$ is itself free. Since all free {\pmp} actions of an amenable group are weakly equivalent \cite{K12}, this recovers what is known in the amenable case. The corollary below in particular implies Theorem \ref{intro:prod} from the introduction.

\begin{cor} \label{cor:prodfree}
Let $G$ be a countably infinite group, let $G \acts (X, \mu)$ be a free {\pmp} action, and let $\cF$ be a $G$-invariant sub-$\sigma$-algebra. Let $G \acts (Y, \nu)$ be a {\pmp} action which is weakly contained in all free {\pmp} actions of $G$.
\begin{enumerate}
\item[\rm (1)] If $\cP$ is a countable partition of $X$ with $\sH(\cP \given \cF) < \infty$ then
$$\rh_{G,\mu}(\cP \given \cF) = \rh_{G, \mu \times \nu}(\cP \given \cF) = \rh_{G, \mu \times \nu}(\cP \given \cF \vee \Borel(Y)).$$
\item[\rm (2)] Without any additional assumptions we have
$$\rh_G(X, \mu \given \cF) = \rh_{G, \mu \times \nu}(\Borel(X) \given \cF) = \rh_G(X \times Y, \mu \times \nu \given \cF \vee \Borel(Y)).$$
\end{enumerate}
\end{cor}

\begin{proof}
As in the proof of the previous theorem, it suffices to show $\rh_{G,\mu \times \nu}(\cP \given \cF \vee \Borel(Y)) \geq \rh_{G,\mu}(\cP \given \cF)$ and $\rh_G(X \times Y, \mu \times \nu \given \cF \vee \Borel(Y)) \geq \rh_G(X, \mu \given \cF)$. Let $G \acts (Y', \nu')$ be a free {\pmp} action which is weakly contained in all other free {\pmp} actions. For example, one can let $(Y', \nu')$ be a (free) Bernoulli shift \cite{AW13}. Since $G \acts (Y', \nu')$ weakly contains $G \acts (Y, \nu)$, it follows that $G \acts (X \times Y', \mu \times \nu')$ weakly contains $G \acts (X \times Y, \mu \times \nu)$ as joinings with $G \acts (X, \mu)$. So by Corollary \ref{cor:wcjoin} and Theorem \ref{thm:prod} we have
$$\rh_{G,\mu}(\cP \given \cF) = \rh_{G, \mu \times \nu'}(\cP \given \cF \vee \Borel(Y')) \leq \rh_{G, \mu \times \nu}(\cP \given \cF \vee \Borel(Y)) \quad \text{and}$$
\begin{equation*}
\rh_G(X, \mu \given \cF) = \rh_G(X \times Y', \mu \times \nu' \given \cF \vee \Borel(Y')) \leq \rh_G(X \times Y, \mu \times \nu \given \cF \vee \Borel(Y)).\qedhere
\end{equation*}
\end{proof}

\begin{cor} \label{cor:prod}
Let $G \acts (X, \mu)$ be an aperiodic {\pmp} action of stabilizer type $\theta$ and let $\cF$ be a $G$-invariant sub-$\sigma$-algebra. For every standard probability space $(L, \lambda)$ we have
$$\rh_G(X, \mu \given \cF) = \rh_G(X \times L^G, \mu \relprod{\theta} \lambda^{\theta \backslash G} \given \cF \vee \Borel(L^G)).$$
Furthermore, $\rh_{G,\mu}(\cP \given \cF) = \rh_{G, \mu \relprod{\theta} \lambda^{\theta \backslash G}}(\cP \given \cF \vee \Borel(L^G))$ for every partition $\cP$ of $X$ with $\sH(\cP \given \cF) < \infty$.
\end{cor}

\begin{proof}
Tucker-Drob \cite{TD12} proved that $G \acts (L^G, \lambda^{\theta \backslash G})$ has stabilizer type $\theta$ and is weakly contained in all {\pmp} actions of stabilizer type $\theta$ (extending a similar result by Ab\'{e}rt--Weiss for free actions \cite{AW13}). Now apply Theorem \ref{thm:prod}.
\end{proof}

In the case of a free action $G \acts (X, \mu)$, the above corollary implies Corollary \ref{intro:bern} from the introduction.

\section{Upper bounds to Rokhlin entropy} \label{sec:drop}

In this section we exhibit expressions which are upper bounds to Rokhlin entropy. Of course, these expressions are also upper bounds to sofic entropy. The formulas we obtain can be approximated in some specific situations, and in the case of a free action of an amenable group they are in fact equal to the classical entropy.

The group $\Z$ with its natural linear order creates a natural notion of past and future for actions of $\Z$. This notion of past plays a fundamental role in the entropy theory for $\Z$-actions. We begin this section by investigating how a linear order, and its induced notion of past, relate to the Rokhlin entropy of actions of non-amenable groups. In this setting, the linear ordering will most likely not come directly from the acting group, but instead will come from a (partial) ordering of the underlying space being acted upon. 

\begin{thm} \label{thm:int}
Let $G \acts (X, \mu)$ be a {\pmp} action and let $\cF$ and $\Sigma$ be $G$-invariant sub-$\sigma$-algebras with $\cF$ countably generated. Let $G \acts (Y, \nu)$ be the factor of $(X, \mu)$ associated to $\Sigma$, and let $\mu = \int \mu_y \ d \nu(y)$ be the disintegration of $\mu$ over $\nu$. Let $r : Y \rightarrow \R$ be a Borel function and for $y \in Y$ set $L_y = \{g \in G : r(g \cdot y) < r(y)\}$. If $\alpha$ is a partition of $X$ satisfying $\sH(\alpha \given \cF \vee \Sigma) < \infty$ then
$$\rh_{G,\mu}(\alpha \given \cF \vee \Sigma) \leq \int_Y \sH_{\mu_y}(\alpha \given \alpha^{L_y} \vee \cF) \ d \nu(y).$$
In particular, if $\salg_G(\alpha) \vee \cF \vee \Sigma \vee \sinv_G = \Borel(X)$ then $\rh_G(X, \mu \given \cF \vee \Sigma)$ is bounded above by this expression.
\end{thm}

\begin{proof}
Let $\pi : (X, \mu) \rightarrow (Y, \nu)$ denote the factor map. Since $\int \sH_{\mu_y}(\alpha \given \cF) \ d \nu(y) = \sH(\alpha \given \cF \vee \Sigma) < \infty$, we see that the functions $y \mapsto \sH_{\mu_y}(\alpha \given \cF)$ and $y \mapsto \sH_{\mu_y}(\alpha \given \alpha^{L_y} \vee \cF)$ are finite almost-everywhere and have finite integral. Fix $\delta > 0$. By the monotone convergence theorem and Lemma \ref{lem:shan} we can fix a finite $T \subseteq G$ with
$$\int_Y \sH_{\mu_y}(\alpha \given \alpha^{T \cap L_y} \vee \cF) \ d \nu(y) < \int_Y \sH_{\mu_y}(\alpha \given \alpha^{L_y} \vee \cF) \ d \nu(y) + \delta.$$
Let $\epsilon > 0$ be such that
$$\nu(C) \cdot \sH_{\pi^{-1}(C)}(\alpha \given \cF \vee \Sigma) = \int_C \sH_{\mu_y}(\alpha \given \cF) \ d \nu(y) < \delta$$
for every Borel set $C \subseteq Y$ having measure less than $\epsilon$. Fix $w > 0$ with
$$\nu( \{y \in Y : \exists t \in T \ 0 < |r(t \cdot y) - r(y)| \leq w\} ) < \frac{\epsilon}{2}.$$
Also fix a compact interval $K \subseteq \R$ satisfying $\nu(r^{-1}(K)) > 1 - \frac{\epsilon}{2}$. Finally, let $\cP = \{P_i : 1 \leq i \leq m\}$ be a finite partition of $K$ such that each $P_i$ is an interval of width less than $w$.

The function $r$ produces a partial ordering of $Y$. We first create a discrete approximation to this partial order. First set
$$Q_0 = \{y \in Y : r(y) \not\in K \text{ or } \exists t \in T \ 0 < |r(t \cdot y) - r(y)| \leq w\},$$
and then for $1 \leq i \leq m$ define
$$Q_i = r^{-1}(P_i) \setminus Q_0.$$
Then $\{Q_i : 0 \leq i \leq m\}$ is a partition of $Y$ and $\nu(Q_0) < \epsilon$. This partition is, with respect to $T$, a good discreet model for the partial order induced by $r : Y \rightarrow \R$ in the sense that for all $y \in Q_i$ with $i > 0$ we have
$$(T \cap L_y) \cdot y \subseteq \bigcup_{0 \leq j < i} Q_j.$$

For each $0 \leq i \leq m$ we set
$$\alpha_i = \Big\{X \setminus \pi^{-1}(Q_i) \Big\} \cup \Big\{ A \cap \pi^{-1}(Q_i) : A \in \alpha \Big\}$$
and $\Psi_i = \salg_G(\alpha_0 \vee \cdots \vee \alpha_i) \vee \cF \vee \Sigma$. Since $\nu(Q_0) < \epsilon$ and the restriction of $\alpha_0$ to $X \setminus \pi^{-1}(Q_0)$ is trivial, our choice of $\epsilon$ gives
\begin{equation} \label{eqn:past1}
\sH(\alpha_0 \given \cF \vee \Sigma) = \nu(Q_0) \cdot \sH_{\pi^{-1}(Q_0)}(\alpha \given \cF \vee \Sigma) < \delta.
\end{equation}
Also note that $\alpha \leq \alpha_0 \vee \cdots \vee \alpha_m \subseteq \Psi_m$.

We claim that for every $1 \leq i \leq m$
\begin{equation} \label{eqn:past3}
\rh_{G,\mu}(\alpha_i \given \Psi_{i-1}) \leq \int_{Q_i} \sH_{\mu_y}(\alpha \given \alpha^{T \cap L_y} \vee \cF) \ d \nu(y).
\end{equation}
As we noted previously, for $i \geq 1$ and $y \in Q_i$ we have
$$(T \cap L_y) \cdot y \subseteq \bigcup_{0 \leq j < i} Q_j.$$
Consider $i \geq 1$, $y \in Q_i$, and $t \in T \cap L_y$. Let $j(t) < i$ be such that $t \cdot y \in Q_{j(t)}$. Then we have
$$t \cdot \Big( (t^{-1} \cdot \alpha) \res \pi^{-1}(y) \Big) = \alpha \res \pi^{-1}(t \cdot y) = \alpha_{j(t)} \res \pi^{-1}(t \cdot y).$$
So
$$(t^{-1} \cdot \alpha) \res \pi^{-1}(y) = (t^{-1} \cdot \alpha_{j(t)}) \res \pi^{-1}(y) \subseteq \Psi_{i-1} \res \pi^{-1}(y).$$
Letting $t \in T \cap L_y$ vary, we conclude that for $y \in Q_i$
$$\alpha^{T \cap L_y} \res \pi^{-1}(y) \subseteq \Psi_{i-1} \res \pi^{-1}(y).$$
Therefore for $1 \leq i \leq m$
\begin{align*}
\rh_{G,\mu}(\alpha_i \given \Psi_{i-1}) & \leq \sH(\alpha_i \given \Psi_{i-1})\\
 & = \int_Y \sH_{\mu_y}(\alpha_i \given \Psi_{i-1}) \ d \nu(y)\\
 & = \int_{Q_i} \sH_{\mu_y}(\alpha \given \Psi_{i-1}) \ d \nu(y)\\
 & \leq \int_{Q_i} \sH_{\mu_y}(\alpha \given \alpha^{T \cap L_y} \vee \cF) \ d \nu(y)
\end{align*}
as claimed.

Since $\alpha \subseteq \Psi_m$, sub-additivity of Rokhlin entropy together with (\ref{eqn:past1}) and (\ref{eqn:past3}) gives
\begin{align*}
\rh_{G,\mu}(\alpha \given \cF \vee \Sigma) & \leq \rh_{G,\mu}(\alpha_0 \given \cF \vee \Sigma) + \sum_{i = 1}^m \rh_{G,\mu}(\alpha_i \given \Psi_{i-1})\\
 & \leq \sH(\alpha_0 \given \cF \vee \Sigma) + \sum_{i = 1}^m \int_{Q_i} \sH_{\mu_y}(\alpha \given \alpha^{T \cap L_y} \vee \cF) \ d \nu(y)\\
 & < \delta + \int_Y \sH_{\mu_y}(\alpha \given \alpha^{T \cap L_y} \vee \cF) \ d \nu(y)\\
 & < \int_Y \sH_{\mu_y}(\alpha \given \alpha^{L_y} \vee \cF) \ d \nu(y) + 2 \delta.
\end{align*}
Now let $\delta$ tend to $0$.
\end{proof}

By combining the previous theorem with Corollary \ref{intro:bern}, we obtain a canonical randomized past which can be used to bound the Rokhlin entropy of any free action. In the case of sofic entropy, the corollary below was independently obtained by Andrei Alpeev and Lewis Bowen (personal communication).

\begin{cor} \label{cor:intfree}
Let $G$ be a countably infinite group, let $G \acts (X, \mu)$ be a free {\pmp} action, and let $\cF$ be a $G$-invariant sub-$\sigma$-algebra. Consider the Bernoulli shift $([0, 1]^G, \lambda^G)$ where $\lambda$ is Lebesgue measure, and for $y \in [0,1]^G$ define $L_y = \{g \in G : y(g^{-1}) < y(1_G)\}$. If $\alpha$ is a partition with $\sH(\alpha \given \cF) < \infty$ then
$$\rh_{G,\mu}(\alpha \given \cF) \leq \int_{[0,1]^G} \sH_\mu(\alpha \given \alpha^{L_y} \vee \cF) \ d \lambda^G(y).$$
In particular, if $\salg_G(\alpha) \vee \cF \vee \sinv_G = \Borel(X)$ then $\rh_G(X, \mu \given \cF)$ is bounded above by this expression.
\end{cor}

The above corollary is a special case of the more general result below which does not require a free action.

\begin{cor} \label{cor:int}
Let $G \acts (X, \mu)$ be an aperiodic {\pmp} action of stabilizer type $\theta$, let $\mu = \int \mu_H \ d \theta(H)$ be the disintegration of $\mu$ over $\theta$, and let $\cF$ be a countably generated $G$-invariant sub-$\sigma$-algebra. For $y \in [0,1]^G$ define $L_y = \{g \in G : y(g^{-1}) < y(1_G)\}$, and let $\lambda$ be Lebesgue measure on $[0, 1]$. If $\alpha$ is a partition with $\sH(\alpha \given \cF) < \infty$ then
$$\rh_{G,\mu}(\alpha \given \cF) \leq \int_{[0,1]^G} \sH_{\mu_{\Stab(y)}}(\alpha \given \alpha^{L_y} \vee \cF) \ d \lambda^{\theta \backslash G}(y).$$
In particular, if $\salg_G(\alpha) \vee \cF \vee \sinv_G = \Borel(X)$ then $\rh_G(X, \mu \given \cF)$ is bounded above by this expression.
\end{cor}

\begin{proof}
Define $r : [0, 1]^G \rightarrow \R$ by $r(y) = y(1_G)$. Using this $r$, the definition of $L_y$ in the statement of the corollary coincides with the definition of $L_y$ in the statement of Theorem \ref{thm:int}. So by Corollary \ref{cor:prod} and Theorem \ref{thm:int} we have
\begin{align*}
\rh_{G,\mu}(\alpha \given \cF) & = \rh_{G, \mu \relprod{\theta} \lambda^{\theta \backslash G}}(\alpha \given \cF \vee \Borel([0,1]^G))\\
 & \leq \int_{[0,1]^G} \sH_{\mu_{\Stab(y)} \times \delta_y}(\alpha \given \alpha^{L_y} \vee \cF) \ d \lambda^{\theta \backslash G}(y)\\
 & = \int_{[0,1]^G} \sH_{\mu_{\Stab(y)}}(\alpha \given \alpha^{L_y} \vee \cF) \ d \lambda^{\theta \backslash G}(y).\qedhere
\end{align*}
\end{proof}

Surprisingly, the upper bound described in Corollary \ref{cor:intfree} coincides with classical Kolmogorov--Sinai entropy for free actions of amenable groups \cite[Theorem 3]{Ki}. However, the next lemma demonstrates the unfortunate truth that this expression is not an isomorphism invariant for actions of non-amenable groups.

\begin{lem} \label{lem:fail}
Let $G$ be a countable non-amenable group. Then there exists a free ergodic {\pmp} action $G \acts (X, \mu)$ which is a factor of a Bernoulli shift and satisfies the following property. For every $r > 0$ there is a partition $\alpha$ with $\salg_G(\alpha) = \Borel(X)$ and
$$r < \int_{[0, 1]^G} \sH_\mu(\alpha \given \alpha^{L_y}) \ d \lambda^G(y) < \infty,$$
where $\lambda$ and $L_y$ are as in Corollary \ref{cor:intfree}. Furthermore, if $G$ is a free group then the action $G \acts (X, \mu)$ can be chosen to be isomorphic to a Bernoulli shift.
\end{lem}

\begin{proof}
We assume that $G$ is finitely generated. The general case will immediately follow from a standard coinduction argument and the observation that every non-amenable group contains a finitely generated non-amenable subgroup. Our argument will involve random graphs. If $(V, E)$ is vertex-transitive graph and $\eta$ is a probability measure on the space of subgraphs of $(V, E)$ which is invariant under the automorphism group of $(V, E)$, then $\int \deg_\Gamma(v) \ d \eta(\Gamma)$ is independent of the choice of $v \in V$ and is called the \emph{expected degree} of $\eta$.

We recall the notion of the minimal spanning forest. Let $G$ be a finitely generated group with finite generating set $S$, let $\Cay(G, S)$ be the associated directed Cayley graph where each $g \in G$ has an outgoing edge to $g s$ for each $s \in S$. Let $E$ be the set of directed edges in $\Cay(G, S)$. For $y \in [0, 1]^E$ define $\Gamma(y)$ to be the subgraph of $\Cay(G, S)$ obtained by removing from each simple (undirected) cycle in $\Cay(G, S)$ all edges having maximum $y$-value within that cycle. Note that $\Gamma(y)$ has no cycles and is thus a forest. The measure $\lambda^E$, where $\lambda$ is Lebesgue measure on $[0, 1]$, pushes forward to a measure $\FMSF(G, S) = \Gamma_*(\lambda^E)$ on the set of forest subgraphs of $\Cay(G, S)$. The measure $\FMSF(G,S)$ is called the \emph{minimal spanning forest} of $\Cay(G, S)$. A theorem of Thom states that if $G$ is non-amenable then the expected degree of $\FMSF(G, S)$ is unbounded as the generating set $S$ varies \cite{Th15, Th}.

We will need a slight modification of the minimal spanning forest. Let $G$ and $S$ be as before and consider $\Cay(G, S)$. For $y \in [0, 1]^G$ we let $\Psi(y)$ be the subgraph of $\Cay(G, S)$ obtained by removing from each simple (undirected) cycle in $\Cay(G, S)$ all edges whose source vertices have maximum $y$-value within that cycle. We define the modified minimal spanning forest to be $\FMSF'(G, S) = \Psi_*(\lambda^G)$. Thom's argument in \cite{Th} is easily modified to show that for each non-amenable group $G$ the expected degree of $\FMSF'(G, S)$ is unbounded as the generating set $S$ varies.

Now consider the Bernoulli shift $G \acts (2^G, u_2^G)$. Note that $2^G$ is a compact abelian group under the product topology and coordinate-wise addition mod $2$, and also note that $u_2^G$ is the unique Haar probability measure on $2^G$. Write $2$ for the set of constant functions in $2^G$ (constantly $0$ and constantly $1$). This is a normal closed (finite) $G$-invariant subgroup of $2^G$. Let $X$ be the compact abelian quotient $2^G / 2$ and let $\mu$ be the Haar probability measure on $X$. Then $G \acts (X, \mu)$ is a {\pmp} action and $\mu$ is the push-forward measure of $u_2^G$ under the quotient map. For a discussion on the Rokhlin entropy of $G \acts (X, \mu)$, see \cite{GS15}.

Fix $r > 0$. Fix a finite generating set $S$ so that the expected degree of $\FMSF'(G, S)$ is greater than $r / \log(2)$. Let $E$ be the set of directed edges in $\Cay(G, S)$. Define $\phi : 2^G \rightarrow 2^E$ by letting $\phi(w)(e)$ be the mod $2$ difference in the $w$-values of the head and tail of $e$. Set $\nu = \phi_*(u_2^G)$. The map $\phi$ is easily seen to be a continuous $G$-equivariant group homomorphism having kernel $2$. Therefore $G \acts (2^E, \nu)$ is isomorphic to $G \acts (X, \mu)$. We will proceed to study $G \acts (2^E, \nu)$. We first comment on the nature of $\nu$. It is easily checked that for any finite subgraph $(V, E')$ of $\Cay(G, S)$ which is a forest, the random variables $z \in 2^E \rightarrow z(e)$, $e \in E'$, have $\nu$-distribution $u_2$ and are $\nu$-independent. However, along any (undirected) cycle in $\Cay(G, S)$ the labels must sum to $0$ mod $2$ $\nu$-almost-always. These two facts completely describe the measure $\nu$.

Let $\alpha$ be the partition of $2^E$ determined by the values of the directed edges originating from $1_G$. Then $\alpha$ is a generating partition. Consider a point $y \in [0, 1]^G$ and the set $L_y = \{g \in G : y(g^{-1}) < y(1_G)\}$. Observe $L_y^{-1} = \{g \in G : y(g) < y(1_G)\}$. Assume that $y : G \rightarrow [0, 1]$ is injective (this is satisfied on a $\lambda^G$-conull set). Note that $\alpha^{L_y}$ is the smallest $\sigma$-algebra for which the maps $z \in 2^E \mapsto z(e)$ are measurable for every directed edge $e$ originating from $L_y^{-1}$. If $E'$ is the set of edges in $\Psi(y)$ which have source vertex $1_G$, then by definition there cannot exist any simple (undirected) cycle in $\Cay(G, S)$ which traverses an edge in $E'$ and only uses vertices in $\{1_G\} \cup L_y^{-1}$. Therefore the random variables $z \in 2^E \mapsto z(e)$, $e \in E'$, are mutually independent and collectively independent of $\alpha^{L_y}$. It follows that $\sH(\alpha \given \alpha^{L_y}) \geq \log(2) \cdot \deg_{\Psi(y)}(1_G)$. Therefore
$$\int_{y \in [0, 1]^G} \sH(\alpha \given \alpha^{L_y}) \ d \lambda^G(y) \geq \log(2) \cdot \int \deg_\Gamma(1_G) \ d \FMSF'(G, S)(\Gamma) > r.$$
Note that the left-most expression is bounded above by $|S| \cdot \log(2) < \infty$.

Finally, when $G$ is a free group, it is easily checked (by letting $S$ be a free generating set above) that $G \acts (X, \mu)$ is isomorphic to a Bernoulli shift. Indeed, this is the well known example of Ornstein and Weiss \cite{OW87}.
\end{proof}

We now include our final upper bound to Rokhlin entropy. It is well known that this upper bound coincides with classical entropy for free actions of amenable groups \cite{DGRS,DFR}. For actions of non-amenable groups it is not an isomorphism invariant, as discovered by Bowen and recorded in \cite{Bu15}.

\begin{thm} \label{thm:drop}
Let $G$ be a countably infinite group, let $G \acts (X, \mu)$ be a free {\pmp} action, and let $\cF$ be a $G$-invariant sub-$\sigma$-algebra. If $\alpha$ is any countable partition then
$$\rh_{G,\mu}(\alpha \given \cF) \leq \inf_{\substack{T \subseteq G\\T \text{ finite}}} \frac{1}{|T|} \cdot \sH(\alpha^T \given \cF \vee \sinv_G).$$
In particular, if $\salg_G(\alpha) \vee \cF \vee \sinv_G = \Borel(X)$ then $\rh_G(X, \mu \given \cF)$ is bounded above by this expression.
\end{thm}

In particular, if $G$ is sofic and $\alpha$ is a generating partition for $G \acts (X, \mu)$ then the sofic entropy is bounded above by $\inf_T \sH(\alpha^T) / |T|$. This bound for sofic entropy was independently discovered by Mikl\'{o}s Ab\'{e}rt, Tim Austin, Lewis Bowen, and Benjy Weiss. Peter Burton also independently obtained a topological version of this upper bound for topological sofic entropy \cite{Bu15}.

\begin{proof}
Since $\rh_{G,\mu}(\alpha \given \cF) = \rh_{G, \mu}(\alpha \given \cF \vee \sinv_G)$, without loss of generality we may assume that $\sinv_G \subseteq \cF$. The claim is vacuously true if $\sH(\alpha \given \cF) = \infty$. So assume that $\sH(\alpha \given \cF) < \infty$. Fix $h > \inf_T \sH(\alpha^T \given \cF) / |T|$. Let $T \subseteq G$ have minimum cardinality with the property that $\sH(\alpha^T \given \cF) / |T| < h$. Since the action of $G$ is measure-preserving, without loss of generality we may suppose that $1_G \in T$. If $P \subseteq T$ and $\varnothing \neq P \neq T$ then we have
$$h > \frac{1}{|T|} \cdot \sH(\alpha^T \given \cF) = \frac{|P|}{|T|} \cdot \frac{1}{|P|} \cdot \sH(\alpha^P \given \cF) + \frac{|T \setminus P|}{|T|} \cdot \frac{1}{|T \setminus P|} \cdot \sH(\alpha^T \given \alpha^P \vee \cF).$$
Since the right-hand side above is a convex combination and $\sH(\alpha^P \given \cF) / |P| \geq h$ by definition of $T$, it follows that
\begin{equation} \label{eqn:drop}
\frac{1}{|T \setminus P|} \cdot \sH(\alpha^T \given \alpha^P \vee \cF) < h
\end{equation}
for all $P \subseteq T$ with $P \neq T$. Define $\phi : [0,1]^G \rightarrow \R$ by $\phi(y) = \sH(\alpha^T \given \alpha^{T L_y} \vee \cF)$, where $L_y = \{g \in G : y(g^{-1}) < y(1_G)\}$. By Corollary \ref{cor:intfree} we have
\begin{equation} \label{eqn:drop2}
\rh_{G,\mu}(\alpha \given \cF) = \rh_{G,\mu}(\alpha^T \given \cF) \leq \int_{[0,1]^G} \phi \ d \lambda^G,
\end{equation}
where $\lambda$ is Lebesgue measure on $[0,1]$.

Let $Y$ be the set of $y \in [0, 1]^G$ such that the map $g \in G \mapsto y(g)$ is injective. Note that $Y$ is $G$-invariant, $G$ acts freely on $Y$, and $\lambda^G(Y) = 1$. For $y \in Y$ let $c(y)$ be the unique element $z$ of $T^{-1} \cdot y$ for which $z(1_G)$ is least. Since $(g \cdot z)(1_G) < z(1_G)$ if and only if $g \in L_z$, it is easy to see that
$$|c^{-1}(z)| = |T \cdot z \setminus T L_z \cdot z| = |T \setminus T L_z|.$$
Note that $T \not\subseteq T L_z$ when $c^{-1}(z) \neq \varnothing$. Define $f : Y \rightarrow \R$ by
$$f(y) = \frac{1}{|c^{-1}(c(y))|} \cdot \phi(c(y)) = \frac{1}{|T \setminus T L_{c(y)}|} \cdot \sH(\alpha^T \given \alpha^{T L_{c(y)}} \vee \cF).$$
By (\ref{eqn:drop}) $f(y) < h$ for all $y$ and hence
\begin{equation} \label{eqn:drop3}
\int_{[0,1]^G} f \ d \lambda^G < h.
\end{equation}

For $t \in T$ set
$$f_t(y) = \begin{cases}
f(y) & \text{if } c(y) = t^{-1} \cdot y\\
0 & \text{otherwise}.
\end{cases}$$
Then $f = \sum_{t \in T} f_t$. For $z \in Y$ note that $\phi(z) = 0$ if $c^{-1}(z) = \varnothing$, as this implies $T \subseteq T L_z$. So we have
$$\phi(z) = \sum_{y \in c^{-1}(z)} f(y) = \sum_{t \in T} f_t(t \cdot z).$$
Therefore by (\ref{eqn:drop2}), (\ref{eqn:drop3}), and $G$-invariance of $\lambda^G$ we have
\begin{align*}
\rh_{G, \mu}(\alpha \given \cF) \leq \int_{[0,1]^G} \phi \ d \lambda^G & = \int_{[0,1]^G} \sum_{t \in T} f_t \circ t \ d \lambda^G\\
 & = \int_{[0,1]^G} \sum_{t \in T} f_t \ d \lambda^G = \int_{[0,1]^G} f \ d \lambda^G < h.
\end{align*}
Now let $h$ tend to $\inf_W \sH(\alpha^W \given \cF) / |W|$.
\end{proof}

\section{Actions of finite-index subgroups} \label{sec:sub}

In the final two sections we discuss two additional applications of weak containment concepts. In this section we look at how Rokhlin entropy behaves for restricted actions of finite-index subgroups. We first observe a simple inequality.

\begin{lem} \label{lem:simple}
Let $G \acts (X, \mu)$ be a {\pmp} action, let $\cF$ be a $G$-invariant sub-$\sigma$-algebra, and let $\Gamma \leq G$ be a finite-index subgroup. Then for the restricted action $\Gamma \acts (X, \mu)$ we have $\rh_\Gamma(X, \mu \given \cF) \leq |G : \Gamma| \cdot \rh_G(X, \mu \given \cF)$.
\end{lem}

\begin{proof}
Let $T \subseteq G$ meet every coset of $\Gamma$ in $G / \Gamma$ precisely once. Then $|T| = |G : \Gamma|$. Fix a countable partition $\alpha$ satisfying $\salg_G(\alpha) \vee \cF \vee \sinv_G = \Borel(X)$. Since $\Gamma \cdot T^{-1} = G$ and $\sinv_\Gamma \supseteq \sinv_G$, we see that $\salg_\Gamma(\alpha^T) \vee \cF \vee \sinv_\Gamma = \Borel(X)$. Therefore, for every such $\alpha$ we have
$$\rh_\Gamma(X, \mu) \leq \sH(\alpha^T \given \cF \vee \sinv_\Gamma) \leq \sH(\alpha^T \given \cF \vee \sinv_G) \leq |G : \Gamma| \cdot \sH(\alpha \given \cF \vee \sinv_G).$$
Taking the infimum over all such $\alpha$ completes the proof.
\end{proof}

In this section we investigate the validity of the ``subgroup formula'' $\rh_\Gamma(X, \mu) = |G : \Gamma| \cdot \rh_G(X, \mu)$ (not surprisingly, for this equality to hold it is necessary that $\Stab(x) \leq \Gamma$ for $\mu$-almost-every $x \in X$; this follows by modifying the proof of Lemma \ref{lem:simple}). A positive answer to this question can be quickly obtained when the ergodic decomposition of $\Gamma \acts (X, \mu)$ has a particular form.

For a Borel space $X$ and a Borel action $G \acts X$, we write $\E_G(X)$ for the set of ergodic $G$-invariant Borel probability measures on $X$. Recall that $\E_G(X)$ has a natural standard Borel structure which is generated by the maps $\nu \mapsto \nu(A)$ for Borel $A \subseteq X$.

\begin{lem} \label{lem:split}
Let $G \acts (X, \mu)$ be a {\pmp} action, and let $\Gamma$ be a finite-index subgroup. Let $\mu = \int_{\E_G(X)} \nu \ d \tau(\nu)$ be the $G$-ergodic decomposition of $\mu$. The following are equivalent.
\begin{enumerate}
\item[\rm (1)] For $\tau$-almost-every $G$-ergodic measure $\nu \in \E_G(X)$, $\nu$ has precisely $|G : \Gamma|$ many $\Gamma$-ergodic components.
\item[\rm (2)] There is a $\Gamma$-invariant Borel partition $\{X_{\Gamma g} : \Gamma g \in \Gamma \backslash G\}$ of $X$ such that $X_{\Gamma g}$ meets almost-every $G$-orbit, $\mu(X_{\Gamma g}) = |G : \Gamma|^{-1}$, and $\sinv_\Gamma = \{X_{\Gamma g} : \Gamma g \in \Gamma \backslash G\} \vee \sinv_G$.
\item[\rm (3)] $G \acts (X, \mu)$ factors onto the finite action $G \acts (G / \Delta, u_{G / \Delta})$, where $\Delta$ is the maximal normal subgroup of $G$ contained in $\Gamma$.
\end{enumerate}
\end{lem}

\begin{proof}
(3) $\Rightarrow$ (2). Say the map is $f : X \rightarrow G / \Delta$. Set $X_{\Gamma g} = f^{-1}(\Gamma g \Delta)$. Then $X_{\Gamma g}$ meets almost-every $G$-orbit and $\mu(X_{\Gamma g}) = |G : \Gamma|^{-1}$. Now set $\Gamma \backslash G / \Delta = \{\Gamma g \Delta : g \in G\}$ and define $\tilde{f}(x) = \Gamma f(x) \in \Gamma \backslash G / \Delta$. Fix $A \in \sinv_\Gamma$. For $x \in X$ define $\phi(x) = \{\tilde{f}(g \cdot x) : g \in G, \ g \cdot x \in A\}$. Then $\phi$ is $\sinv_G$-measurable and
$$A = \bigcup_{S \subseteq \Gamma \backslash G / \Delta} \Big(\tilde{f}^{-1}(S) \cap \phi^{-1}(S) \Big) \in \{X_{\Gamma g} : \Gamma g \in \Gamma \backslash G\} \vee \sinv_G.$$

(2) $\Rightarrow$ (1). For $\tau$-almost-every $\nu \in \E_G(X)$, we must have that each set $X_{\Gamma g}$ meets $\nu$-almost-every $G$-orbit. Fix $\nu \in \E_G(X)$ with this property. Then each $X_{\Gamma g}$ is $\Gamma$ invariant and satisfies $\nu(X_{\Gamma g}) > 0$, implying that $\nu$ has at least $|G : \Gamma|$ many $\Gamma$-ergodic components. On the other hand, if $\{Y_i : 1 \leq i \leq n\}$ is a Borel partition of a $\nu$-conull subset of $X$ into $\Gamma$-invariant sets of positive measure, then by ergodicity each $Y_i$ must meet $\nu$-almost-every $G$-orbit. Each $G$-orbit contains at most $|G : \Gamma|$ many $\Gamma$-orbits, whence $n \leq |G : \Gamma|$. We conclude that $\nu$ has precisely $|G : \Gamma|$ many $\Gamma$-ergodic components.

(1) $\Rightarrow$ (3). Remove a $\mu$-null set if necessary so that every $\nu \in \E_G(X)$ has $|G : \Gamma|$ many $\Gamma$-ergodic components. Define $f : \E_\Gamma(X) \rightarrow \E_G(X)$ by $f(\eta) = |G : \Gamma|^{-1} \cdot \sum_{g \Gamma \in G / \Gamma} g \cdot \eta$. Then $f$ is Borel and for $\nu \in \E_G(X)$, $f^{-1}(\nu)$ is the set of $\Gamma$-ergodic components of $\nu$. In particular, $|f^{-1}(\nu)| = |G : \Gamma|$. It follows that there are $|G : \Gamma|$-many one-sided Borel inverses (which we choose to index by $\Gamma \backslash G$) $\bar{f}_{\Gamma g} : \E_G(X) \rightarrow \E_\Gamma(X)$ satisfying $f^{-1}(\nu) = \{\bar{f}_{\Gamma g}(\nu) : \Gamma g \in \Gamma \backslash G\}$. Set $\mu_{\Gamma g} = \int \bar{f}_{\Gamma g}(\nu) \ d \tau(\nu)$. Then $\mu_{\Gamma g}$ and $\mu_{\Gamma u}$ are mutually singular whenever $\Gamma g \neq \Gamma u$. Thus there is a $\Gamma$-invariant Borel partition $\{X_{\Gamma g} : \Gamma g \in \Gamma \backslash G\}$ satisfying $\mu_{\Gamma g}(X_{\Gamma g}) = 1$ for every $\Gamma g \in \Gamma \backslash G$. Each $X_{\Gamma g}$ is $\Gamma$-invariant and meets almost-every orbit. It follows that for almost-every $x \in X$ and every $u \in G$, $x$ and $u \cdot x$ lie in the same piece of the partition $\{X_{\Gamma g} : \Gamma g \in \Gamma \backslash G\}$ if and only if $u \in \Gamma$. Therefore $x$ and $u \cdot x$ lie in the same piece of the partition $\{s \cdot X_{\Gamma g} : \Gamma g \in \Gamma \backslash G\}$ if and only if $u \in s \Gamma s^{-1}$. As $\Delta = \bigcap_{s \in G} s \Gamma s^{-1}$, it follows that $G \acts (G / \Delta, u_{G / \Delta})$ is the factor of $G \acts (X, \mu)$ associated to $\salg_G(\{X_{\Gamma g} : \Gamma g \in \Gamma \backslash G\})$.
\end{proof}

Now we check that when the equivalent conditions of the previous lemma are met, Rokhlin entropy satisfies a subgroup formula.

\begin{lem} \label{lem:subform}
Let $G \acts (X, \mu)$ be an aperiodic {\pmp} action, let $\cF$ be a $G$-invariant sub-$\sigma$-algebra, and let $\Gamma$ be a subgroup of finite index. If almost-every $G$-ergodic component decomposes into $|G : \Gamma|$-many $\Gamma$-ergodic components, then
$$\rh_\Gamma(X, \mu \given \cF) = |G : \Gamma| \cdot \rh_G(X, \mu \given \cF).$$
\end{lem}

\begin{proof}
By Lemma \ref{lem:simple} we only need to show $\rh_\Gamma(X, \mu \given \cF) \geq |G : \Gamma| \cdot \rh_G(X, \mu \given \cF)$. Since $G$ acts aperiodically and $\Gamma$ is of finite-index, $\Gamma$ must act aperiodically. So by \cite[Lem. 9.5]{AS} we have $\rh_{G, \mu}(\sinv_\Gamma) \leq \rh_{\Gamma, \mu}(\sinv_\Gamma) = 0$.

Fix a partition $\beta$ satisfying $\salg_\Gamma(\beta) \vee \cF \vee \sinv_\Gamma = \Borel(X)$. Let $\{X_{\Gamma g} : \Gamma g \in \Gamma \backslash G\}$ be the $\Gamma$-invariant partition given by Lemma \ref{lem:split}. Write $\mu_{\Gamma g}$ for the normalized restriction of $\mu$ to $X_{\Gamma g}$. Then
$$\sum_{\Gamma g \in \Gamma \backslash G} \mu(X_{\Gamma g}) \cdot \sH_{\mu_{\Gamma g}}(\beta \given \cF \vee \sinv_\Gamma) = \sH(\beta \given \cF \vee \sinv_\Gamma).$$
So there is $\Gamma u$ with $\mu(X_{\Gamma u}) \cdot \sH_{\mu_{\Gamma u}}(\beta \given \cF \vee \sinv_\Gamma) \leq |G : \Gamma|^{-1} \cdot \sH(\beta \given \cF \vee \sinv_\Gamma)$. Set $\alpha = (\beta \res X_{\Gamma u}) \cup \{X \setminus X_{\Gamma u}\}$. Since $\sinv_\Gamma = \{X_{\Gamma g} : \Gamma g \in \Gamma \backslash G\} \vee \sinv_G$ we have that
$$\salg_G(\alpha) \vee \cF \vee \sinv_G \supseteq \Big( \salg_\Gamma(\beta) \vee \cF \vee \sinv_\Gamma \Big) \res X_{\Gamma u} = \Borel(X) \res X_{\Gamma u}.$$
Therefore $\salg_G(\alpha) \vee \cF \vee \sinv_G = \Borel(X)$. Thus by sub-additivity of Rokhlin entropy
\begin{align*}
\rh_G(X, \mu \given \cF) & \leq \rh_{G,\mu}(\sinv_\Gamma) + \sH(\alpha \given \cF \vee \sinv_\Gamma)\\
 & = \mu(X_{\Gamma u}) \cdot \sH_{\mu_{\Gamma u}}(\beta \given \cF \vee \sinv_\Gamma) \leq \frac{1}{|G : \Gamma|} \cdot \sH(\beta \given \cF \vee \sinv_\Gamma).
\end{align*}
Now take the infimum over $\beta$.
\end{proof}

As a corollary we obtain a lower bound to $\rh_\Gamma(X, \mu)$. Recall that an upper bound is given in Lemma \ref{lem:simple}.

\begin{cor} \label{cor:lower}
Let $G \acts (X, \mu)$ be an aperiodic {\pmp} action and let $\cF$ be a $G$-invariant sub-$\sigma$-algebra. Let $\Gamma \leq G$ be a finite-index subgroup and set $\Delta = \bigcap_{g \in G} g \Gamma g^{-1}$. Then
$$|G : \Gamma| \cdot \rh_G(X \times (G / \Delta), \mu \times u_{G / \Delta} \given \cF \vee \Borel(G / \Delta)) \leq \rh_\Gamma(X, \mu \given \cF),$$
with equality if $\Gamma = \Delta$ is normal.
\end{cor}

\begin{proof}
By using Lemma \ref{lem:split}.(3) and applying Lemma \ref{lem:subform} we obtain
\begin{align*}
|G  & : \Gamma| \cdot \rh_G(X \times (G / \Delta), \mu \times u_{G / \Delta} \given \cF \vee \Borel(G / \Delta))\\
 & = \rh_\Gamma(X \times (G / \Delta), \mu \times u_{G / \Delta} \given \cF \vee \Borel(G / \Delta)).
\end{align*}
From the definition of relative Rokhlin entropy we see that the last expression is at most $\rh_\Gamma(X, \mu \given \cF)$.

Now suppose that $\Gamma = \Delta$. By the previous paragraph, it suffices to show that
$$\rh_\Delta(X \times (G / \Delta), \mu \times u_{G / \Delta} \given \cF \vee \Borel(G / \Delta)) = \rh_\Delta(X, \mu \given \cF).$$
For each $g \Delta \in G / \Delta$, the actions $\Delta \acts (X, \mu)$ and $\Delta \acts (X \times (G / \Delta), \mu \times \delta_{g \Delta})$ are isomorphic. Since Rokhlin entropy is an affine function on the space of invariant measures \cite{AS} and $\sinv_\Delta \supseteq \Borel(G / \Delta)$, we have
\begin{equation*}
\rh_\Delta(X, \mu \given \cF) = \rh_\Delta(X \times (G / \Delta), \mu \times u_{G / \Delta} \given \cF) = \rh_\Delta(X \times (G / \Delta), \mu \times u_{G / \Delta} \given \cF \vee \Borel(G / \Delta)).
\end{equation*}
\end{proof}

We can now begin to see a connection between the entropy of restricted actions of subgroups and preservation of Rokhlin entropy under direct products. Recall that an action $G \acts (X, \mu)$ is \emph{finite} if there is a normal finite-index subgroup $\Delta \leq G$ such that $\Delta$ fixes every point in $X$. More generally, an action is called \emph{finitely modular} if it is an inverse limit of finite actions.

\begin{lem} \label{lem:ergfin}
Let $G \acts (X, \mu)$ be an aperiodic {\pmp} action, and let $\cF$ be a $G$-invariant sub-$\sigma$-algebra. For every finitely modular action $G \acts (Y, \nu)$ we have
\begin{align*}
\rh_G(X \times Y, \mu \times \nu \given \cF \vee \Borel(Y)) & \geq \inf_{\substack{\Delta \lhd G\\|G : \Delta| < \infty}} \rh_G(X \times (G / \Delta), \mu \times u_{G / \Delta} \given \cF \vee \Borel(G / \Delta))\\
 & = \inf_{\substack{\Delta \lhd G\\|G : \Delta| < \infty}} \frac{1}{|G : \Delta|} \cdot \rh_\Delta(X, \mu \given \cF).
\end{align*}
\end{lem}

\begin{proof}
The equality of the last two terms is immediate from Corollary \ref{cor:lower}. Fix a finitely modular action $G \acts (Y, \nu)$. Express $G \acts (Y, \nu)$ as the inverse limit of finite actions $G \acts (Y_n, \nu_n)$. It is an easy consequence of the definitions that $G \acts (X \times Y_n, \mu \times \nu_n)$ converges to $G \acts (X \times Y, \mu \times \nu)$ in the topology on the space of joinings with $G \acts (X, \mu)$. So Corollary \ref{cor:upjoin} gives
$$\rh_G(X \times Y, \mu \times \nu \given \cF \vee \Borel(Y)) \geq \limsup_{n \rightarrow \infty} \rh_G(X \times Y_n, \mu \times \nu_n \given \cF \vee \Borel(Y_n)).$$
So we may assume that $G \acts (Y, \nu)$ is a finite action. Moreover, since Rokhlin entropy is an affine function on the space of $G$-invariant probability measures we may assume that $\nu$ is ergodic and $G$ acts transitively on $Y$. Set $\Delta = \bigcap_{y \in Y} \Stab(y)$ and note that $\Delta$ is a normal finite-index subgroup of $G$. By picking any $y_0 \in Y$, sending $\Delta$ to $y_0$, and extending equivariantly, we see that $G \acts (G / \Delta, u_{G / \Delta})$ factors onto $G \acts (Y, \nu)$. Therefore $G \acts (X \times (G / \Delta), \mu \times u_{G / \Delta})$ weakly contains $G \acts (X \times Y, \mu \times \nu)$ as joinings with $G \acts (X, \mu)$. So by Corollary \ref{cor:wcjoin}.(4)
\begin{equation*}
\rh_G(X \times (G / \Delta), \mu \times u_{G / \Delta} \given \cF \vee \Borel(G / \Delta)) \leq \rh_G(X \times Y, \mu \times \nu \given \cF \vee \Borel(Y)).\qedhere
\end{equation*}
\end{proof}

We now present the two main theorems of this section.

\begin{thm} \label{thm:subprod}
Let $G \acts (X, \mu)$ be an aperiodic {\pmp} action, let $\cF$ be a $G$-invariant sub-$\sigma$-algebra, and assume $\rh_G(X, \mu \given \cF) < \infty$. The following are equivalent.
\begin{enumerate}
\item[\rm (1)] $\rh_G(X, \mu \given \cF) = \rh_G(X \times Y, \mu \times \nu \given \cF \vee \Borel(Y))$ for every finitely modular action $G \acts (Y, \nu)$.
\item[\rm (2)] $\rh_\Gamma(X, \mu \given \cF) = |G : \Gamma| \cdot \rh_G(X, \mu \given \cF)$ for every finite-index subgroup $\Gamma \leq G$.
\end{enumerate}
\end{thm}

\begin{proof}
The implication (1) $\Rightarrow$ (2) is immediate from Lemma \ref{lem:simple} and Corollary \ref{cor:lower}. For the implication (2) $\Rightarrow$ (1), one inequality is immediate from definitions and the other follows directly from Lemma \ref{lem:ergfin}.
\end{proof}

\begin{thm} \label{thm:sofic}
Let $G$ be a sofic group with sofic approximation $\Sigma$, and let $G \acts (X, \mu)$ be an aperiodic {\pmp} action. Assume that for all finite-index normal subgroups $\Delta \lhd G$ we have $h_G^\Sigma(G / \Delta, u_{G / \Delta}) \neq - \infty$. Then
\begin{enumerate}
\item[\rm (1)] for every finitely modular action $G \acts (Y, \nu)$
$$h_G^\Sigma(X, \mu) \leq \rh_G(X \times Y, \mu \times \nu \given \Borel(Y)) \leq \rh_G(X, \mu)$$
\item[\rm (2)] for every finite-index subgroup $\Gamma \leq G$
$$|G : \Gamma| \cdot h_G^\Sigma(X, \mu) \leq \rh_\Gamma(X, \mu) \leq |G : \Gamma| \cdot \rh_G(X, \mu)$$
\end{enumerate}
\end{thm}

\begin{proof}
Say $\Sigma = (\sigma_n : G \rightarrow \Sym(d_n))_{n \in \N}$. Fix a finite-index normal subgroup $\Delta \lhd G$. We first claim that there is a joining $\lambda$ of $\mu$ and $u_{G / \Delta}$ with $h_G^\Sigma(X \times (G / \Delta), \lambda) \geq h_G^\Sigma(X, \mu)$ (one can show that any joining $\lambda$ must satisfy the reverse inequality, but we won't need this). Our argument is inspired by the proof of the variational principle \cite{KL11a}. Without loss of generality, we may assume that $X$ is a compact metric space and that $G$ acts continuously on $X$. Let $\rho^X$ be a metric on $X$ giving the topology, let $\rho^{G / \Delta}$ be the discrete metric on $G / \Delta$, and let $\rho$ be the metric on $X \times G / \Delta$ given by $\rho((x, g \Delta), (y, h \Delta)) = \rho^X(x, y) + \rho^{G / \Delta}(g \Delta, h \Delta)$. Our assumption that $h_G^\Sigma(G / \Delta, u_{G / \Delta}) \neq - \infty$ implies there is a sequence $\psi_n : \{0, \ldots, d_n-1\} \rightarrow G / \Delta$ such that for every finite $T \subseteq G$, for every open neighborhood $U$ of $u_{G / \Delta}$, and every $\delta > 0$ we have $\psi_n \in \Map(\rho^{G / \Delta}, T, U, \delta, \sigma_n)$ for all sufficiently large $n \in \N$.

Write $J$ for the set of all joinings $\lambda$ of $\mu$ and $u_{G / \Delta}$. The set $J$ is compact in the weak$^*$-topology. Fix $\epsilon > 0$, and let us first show that there is $\lambda \in J$ with $h_G^{\Sigma, \epsilon}(X \times (G / \Delta), \lambda) \geq h_G^{\Sigma, \epsilon}(X, \mu)$. Towards a contradiction, suppose not. By compactness of $J$ and monotonicity properties of the sets $\Map(\cdot)$, there is a finite collection $\mathcal{U} = \{U_i : i \in I\}$ of open subsets of $\Prob(X \times (G / \Delta))$ which cover $J$, a finite set $T \subseteq G$, and $\delta > 0$ such that
$$\forall i \in I \quad \limsup_{n \rightarrow \infty} \frac{1}{d_n} \cdot \log N_\epsilon(\Map(\rho, T, U_i, \delta, \sigma_n), \rho_2) < h_G^{\Sigma, \epsilon}(X, \mu).$$
However, if we let $V$ be an open neighborhood of $\mu$ satisfying $V \times u_{G / \Delta} \subseteq \bigcup_{i \in I} U_i$ then for sufficiently large $n \in \N$ we have
$$\Map(\rho^X, T, V, \delta / 2, \sigma_n) \times \psi_n \subseteq \bigcup_{i \in I} \Map(\rho, T, U_i, \delta, \sigma_n).$$
Therefore for sufficiently large $n \in \N$
$$N_\epsilon(\Map(\rho^X, T, V, \delta / 2, \sigma_n) \times \psi_n, \rho_2^X) \leq \sum_{i \in I} N_\epsilon(\Map(\rho, T, U_i, \delta, \sigma_n), \rho_2).$$
This is a contradiction since
$$\limsup_{n \rightarrow \infty} \frac{1}{d_n} \cdot \log N_\epsilon(\Map(\rho^X, T, V, \delta / 2, \sigma_n) \times \psi_n, \rho_2^X) \geq h_G^{\Sigma, \epsilon}(X, \mu).$$
We conclude that the set
$$J_\epsilon = \{\lambda \in J : h_G^{\Sigma, \epsilon}(X \times (G / \Delta), \lambda) \geq h_G^{\Sigma, \epsilon}(X, \mu)\}$$
is non-empty. It quickly follows from the definitions that $J_\epsilon$ is also closed. Therefore $J_0 = \bigcap_{n \in \N} J_{1/n}$ is non-empty by compactness, and if $\lambda \in J_0$ then $h_G^\Sigma(X \times (G / \Delta), \lambda) \geq h_G^\Sigma(X, \mu)$. This completes the claim.

Since sofic entropy is a lower bound to Rokhlin entropy \cite{AS,B10b}, the above claim gives a joining $\lambda$ such that $\rh_G(X \times (G / \Delta), \lambda) \geq h_G^\Sigma(X, \mu)$. Writing $\lambda = \frac{1}{|G : \Delta|} \cdot \sum_{g \Delta \in G / \Delta} \lambda_{g \Delta} \times \delta_{g \Delta}$, we have that $\mu$ is the average of the $\Delta$-invariant measures $\lambda_{g \Delta}$. Since Rokhlin entropy is an affine function on the space of $\Delta$-invariant probability measures \cite{AS} and since $\Delta \acts (X, \lambda_{g \Delta})$ is isomorphic to $\Delta \acts (X \times (G / \Delta), \lambda_{g \Delta} \times \delta_{G \Delta})$ we have that
\begin{align*}
\rh_\Delta(X, \mu) & = \frac{1}{|G : \Delta|} \cdot \sum_{g \Delta \in G / \Delta} \rh_\Delta(X, \lambda_{g \Delta})\\
 & = \frac{1}{|G : \Delta|} \cdot \sum_{g \Delta \in G / \Delta} \rh_\Delta(X \times (G / \Delta), \lambda_{g \Delta} \times \delta_{g \Delta})\\
 & = \rh_\Delta(X \times (G / \Delta), \lambda).
\end{align*}
By applying Lemma \ref{lem:subform} we obtain $\rh_G(X \times (G / \Delta), \lambda) = |G : \Delta|^{-1} \cdot \rh_\Delta(X, \mu)$. By repeating this argument with $\mu \times u_{G / \Delta}$ in place of $\lambda$, we obtain
$$\rh_G(X \times (G / \Delta), \mu \times u_{G / \Delta}) = |G : \Delta|^{-1} \cdot \rh_\Delta(X, \mu) = \rh_G(X \times (G / \Delta), \lambda) \geq h_G^\Sigma(X, \mu).$$
Since $|G : \Delta| < \infty$, $\Delta$ must act aperiodically on $X$ and thus $\rh_{G, \mu \times u_{G / \Delta}}(\Borel(G / \Delta)) \leq \rh_{\Delta, \mu \times u_{G / \Delta}}(\Borel(G / \Delta)) = 0$ by \cite[Lem. 9.5]{AS} (the final equality uses $\Borel(G / \Delta) \subseteq \sinv_\Delta$). So sub-additivity of Rokhlin entropy implies that conditioning on $\Borel(G / \Delta)$ has no effect on Rokhlin entropy and thus
$$\rh_G(X \times (G / \Delta), \mu \times u_{G / \Delta} \given \Borel(G / \Delta)) = \rh_G(X \times (G / \Delta), \mu \times u_{G / \Delta}) \geq h_G^\Sigma(X, \mu).$$
This holds for every finite-index normal subgroup $\Delta \lhd G$, so applying Corollary \ref{cor:lower} and Lemma \ref{lem:ergfin} completes the proof.
\end{proof}

Recall that a countable residually finite group $G$ is said to have \emph{property MD} if the finitely modular actions are dense in the space of {\pmp} actions of $G$ with the weak topology \cite{K12}. Equivalently, $G$ has property MD if there is a finitely modular action which weakly contains all other {\pmp} actions of $G$ \cite[Prop. 4.8]{K12}. It is known that all residually finite amenable groups, all free groups, all free products of finite groups, all surface groups, and all fundamental groups of closed hyperbolic $3$-manifolds have property MD \cite{K12,BTD} (the last example relies upon Agol's virtual fibering theorem \cite{Ag13}). Also, property MD is preserved under passage to subgroups and extensions by residually finite amenable groups \cite{K12,BTD}. It is known that $SL_n(\Z)$ has property MD precisely when $n = 2$ ($SL_n(\Z)$ does not have property FD for $n > 2$ \cite{LS04} and thus does not have property MD \cite{K12}).

The main feature of the following corollary is that it discusses product actions $G \acts (X \times Y, \mu \times \nu)$ where $G \acts (Y, \nu)$ varies over all {\pmp} actions, rather than only the finitely modular actions.

\begin{cor}
Let $G$ be a residually finite group with property MD, let $G \acts (X, \mu)$ be an aperiodic {\pmp} action, and let $\cF$ be a $G$-invariant sub-$\sigma$-algebra with $\rh_G(X, \mu \given \cF) < \infty$. The following are equivalent.
\begin{enumerate}
\item[\rm (1)] $\rh_G(X, \mu \given \cF) = \rh_G(X \times Y, \mu \times \nu \given \cF \vee \Borel(Y))$ for all {\pmp} actions $G \acts (Y, \nu)$.
\item[\rm (2)] $\rh_\Gamma(X, \mu \given \cF) = |G : \Gamma| \cdot \rh_G(X, \mu \given \cF)$ for every finite-index subgroup $\Gamma \leq G$.
\end{enumerate}
Furthermore, if $\Sigma$ is a sofic approximation to $G$ with $h_G^\Sigma(G / \Delta, u_{G / \Delta}) \neq - \infty$ for every finite-index normal subgroup $\Delta \lhd G$ and $h_G^\Sigma(X, \mu) = \rh_G(X, \mu) < \infty$, then (1) and (2) hold with $\cF = \{\varnothing, X\}$.
\end{cor}

\begin{proof}
(1) implies Theorem \ref{thm:subprod}.(1) which implies (2). Now assume (2). Let $G \acts (Z, \eta)$ be a finitely modular action which weakly contains all other {\pmp} actions of $G$, and let $G \acts (Y, \nu)$ be any {\pmp} action of $G$. Then $G \acts (X \times Z, \mu \times \eta)$ weakly contains $G \acts (X \times Y, \mu \times \nu)$ as joinings with $G \acts (X, \mu)$. So Theorem \ref{thm:subprod} and Corollary \ref{cor:wcjoin} imply
$$\rh_G(X, \mu \given \cF) = \rh_G(X \times Z, \mu \times \eta \given \cF \vee \Borel(Z)) \leq \rh_G(X \times Y, \mu \times \nu \given \cF \vee \Borel(Y)) \leq \rh_G(X, \mu \given \cF).$$
This proves (1). The final statement follows from Theorem \ref{thm:sofic}.
\end{proof}

\section{Pinsker algebras} \label{sec:pinsker}

In this section we investigate whether the outer Pinsker algebra of a direct product is the join of the outer Pinsker algebras of the factors. For an action $G \acts (X, \mu)$ and a sub-$\sigma$-algebra $\cF$, we write $\Pi(\mu \given \cF)$ for the outer Rokhlin Pinsker algebra of $X$ relative to $\cF$
$$\Pi(\mu \given \cF) = \{A \subseteq X : \rh_{G,\mu}(\{A, X \setminus A\} \given \cF) = 0\}.$$
It follows from the countable sub-additivity of Rokhlin entropy that $\Pi(\mu \given \cF)$ is the largest $G$-invariant sub-$\sigma$-algebra which contains $\cF$ and satisfies $\rh_{G,\mu}(\Pi(\mu \given \cF) \given \cF) = 0$.

\begin{lem} \label{lem:mg}
Let $G \acts (X, \mu)$ be an aperiodic {\pmp} action and let $\cF$ be a $G$-invariant sub-$\sigma$-algebra. Then for every $\epsilon > 0$ there is a partition $\xi$ of $X$ with $\sH(\xi) < \epsilon$ and $\Pi(\mu \given \cF) \subseteq \salg_G(\xi) \vee \cF$.
\end{lem}

\begin{proof}
This would be by definition if the conclusion were $\sH(\xi \given \cF \vee \sinv_G) < \epsilon$ and $\Pi(\mu \given \cF) \subseteq \salg_G(\xi) \vee \cF \vee \sinv_G$. The stated claim follows immediately from \cite[Cor. 5.3]{AS} since $\rh_{G,\mu}(\Pi(\mu \given \cF) \given \cF) = 0$.
\end{proof}

The following lemma roughly says that a weak containment of joinings implies an inequality in the size of the relative outer Pinsker algebras. Of course, the two Pinsker algebras being compared reside in different spaces and thus the notion of how one is larger than the other is a bit subtle.

\begin{lem} \label{lem:cpejoin}
Let $G$ be a countable group, let $G \acts^a (X, \mu)$ and $G \acts^{b_i} (Y_i, \nu_i)$, $i = 1, 2$, be {\pmp} actions, and let $\lambda_i$ be a joining of $a$ with $b_i$. Also let $\cF \subseteq \Borel(X)$ be a $G$-invariant sub-$\sigma$-algebra. Assume that $(b_1, \lambda_1)$ weakly contains $(b_2, \lambda_2)$ as joinings with $a$ and that both actions $G \acts^{a \times b_i} (X \times Y_i, \lambda_i)$ are aperiodic. Then for every partition $\cP$ of $X$ with $\sH(\cP) < \infty$
\begin{enumerate}
\item[\rm (1)] $\sH_{\lambda_1}(\cP \given \Pi(\lambda_1 \given \cF)) \leq \sH_{\lambda_2}(\cP \given \Pi(\lambda_2 \given \cF))$, and
\item[\rm (2)] $\sH_{\lambda_1}(\cP \given \Pi(\lambda_1 \given \cF \vee \Borel(Y_1))) \leq \sH_{\lambda_2}(\cP \given \Pi(\lambda_2 \given \cF \vee \Borel(Y_2)))$.
\end{enumerate}
\end{lem}

\begin{proof}
(2). Set $\Pi_1 = \Pi(\lambda_1 \given \cF \vee \Borel(Y_1))$ and $\Pi_2 = \Pi(\lambda_2 \given \cF \vee \Borel(Y_2))$. Fix a partition $\cP$ of $X$ with $\sH(\cP) < \infty$. Fix $n \in \N$. By Lemma \ref{lem:mg} we may let $\gamma_2$ be a partition of $X \times Y_2$ with $\sH_{\lambda_2}(\gamma_2) < 1 / 2^n$ and with $\Pi_2 \subseteq \salg_{a \times b_2(G)}(\gamma_2) \vee \cF \vee \Borel(Y_2)$. Pick a finite $T \subseteq G$, finite $\xi \subseteq \cF$, and finite $\chi_2 \subseteq \Borel(Y_2)$ with
\begin{equation} \label{eqn:wccpe2}
\sH_{\lambda_2}(\cP \given \gamma_2^{a \times b_2(T)} \vee \xi \vee \chi_2) < \sH_{\lambda_2}(\cP \given \Pi_2) + 1 / 2^n.
\end{equation}
Pick finite labeled partitions $\alpha \subseteq \Borel(X)$ and $\zeta_2 \subseteq \Borel(Y_2)$ and a coarsening $\beta_2^n \leq \alpha \vee \zeta_2$ with $\dR_{\lambda_2}(\beta_2^n, \gamma_2) < 1 / (2^n \cdot 2 |T|)$. Then we have
\begin{equation} \label{eqn:wccpe3}
\sH_{\lambda_2}(\cP \given (\beta_2^n)^{a \times b_2(T)} \vee \xi \vee \chi_2) < \sH_{\lambda_2}(\cP \given \gamma_2^{a \times b_2(T)} \vee \xi \vee \chi_2) + 1 / 2^n
\end{equation}
and also $\sH_{\lambda_2}(\beta_2^n) < \sH_{\lambda_2}(\gamma_2) + 1 / 2^n < 2 / 2^n$.

Fix $\kappa > 0$ to be specified in a moment. Since $(b_1, \lambda_1)$ weakly contains $(b_2, \lambda_2)$ as joinings with $a$, there are labeled partitions $\zeta_1, \chi_1 \subseteq \Borel(Y_1)$ satisfying
$$|\dist_{\lambda_1}(\cP \vee \xi \vee \alpha^{a(T)} \vee \zeta_1^{b_1(T)} \vee \chi_1) - \dist_{\lambda_2}(\cP \vee \xi \vee \alpha^{a(T)} \vee \zeta_2^{b_2(T)} \vee \chi_2)| < \kappa.$$
Using the natural correspondence between the labeled partitions $\zeta_1$ and $\zeta_2$, we can build $\beta_1^n \leq \alpha \vee \zeta_1$ as $\beta_2^n$ is built from $\alpha \vee \zeta_2$. Since $(\beta_1^n)^{a \times b_1(T)} \leq \alpha^{a(T)} \vee \zeta_1^{b_1(T)}$ and $(\beta_2^n)^{a \times b_2(T)} \leq \alpha^{a(T)} \vee \zeta_2^{b_2(T)}$, we see that
$$|\dist_{\lambda_1}(\cP \vee \xi \vee (\beta_1^n)^{a \times b_1(T)} \vee \chi_1) - \dist_{\lambda_2}(\cP \vee \xi \vee (\beta_2^n)^{a \times b_2(T)} \vee \chi_2)| < \kappa.$$
So for sufficiently small $\kappa$ we have
$$\sH_{\lambda_1}(\cP \given (\beta_1^n)^{a \times b_1(T)} \vee \xi \vee \chi_1) < \sH_{\lambda_2}(\cP \given (\beta_2^n)^{a \times b_2(T)} \vee \xi \vee \chi_2) + 1 / 2^n$$
and also
$$\sH_{\lambda_1}(\beta_1^n) < \sH_{\lambda_2}(\beta_2^n) + 1 / 2^n < 3 / 2^n.$$
It follows from (\ref{eqn:wccpe2}) and (\ref{eqn:wccpe3}) that
\begin{equation} \label{eqn:wccpe}
\sH_{\lambda_1}(\cP \given \salg_{a \times b_1(G)}(\beta_1^n) \vee \cF \vee \Borel(Y_1)) < \sH_{\lambda_2}(\cP \given \Pi_2) + 3 / 2^n.
\end{equation}

Consider $\Sigma = \bigcap_{n \in \N} \bigvee_{k \geq n} \salg_{a \times b_1(G)}(\beta_1^k) \vee \cF \vee \Borel(Y_1)$. We have
\begin{align*}
\rh_{G,\lambda_1}(\Sigma \given \cF \vee \Borel(Y_1)) & \leq \inf_{n \in \N} \rh_{G,\lambda_1} \left( \bigvee_{k \geq n} \beta_1^k \Given \cF \vee \Borel(Y_1) \right)\\
 & \leq \inf_{n \in \N} \sum_{k \geq n} \sH(\beta_1^k)\\
 & \leq \inf_{n \in \N} \sum_{k \geq n} 3 / 2^k\\
 & = 0.
\end{align*}
Therefore $\Sigma \subseteq \Pi_1$. Now by (\ref{eqn:wccpe}) we have
\begin{align*}
\sH_{\lambda_1}(\cP \given \Pi_1) & \leq \sH_{\lambda_1}(\cP \given \Sigma)\\
 & = \lim_{n \rightarrow \infty} \sH_{\lambda_1} \left( \cP \Given \bigvee_{k \geq n} \salg_{a \times b_1(G)}(\beta_1^k) \vee \cF \vee \Borel(Y_1) \right)\\
 & \leq \liminf_{n \rightarrow \infty} \sH_{\lambda_1}(\cP \given \salg_{a \times b_1(G)}(\beta_1^n) \vee \cF \vee \Borel(Y_1))\\
 & \leq \sH_{\lambda_2}(\cP \given \Pi_2).
\end{align*}

(1). The proof is identical up to excluding the partitions $\chi_1$ and $\chi_2$ and excluding $\Borel(Y_1)$ and $\Borel(Y_2)$ from certain expressions.
\end{proof}

In order to use the previous lemma, we observe a simple fact.

\begin{lem} \label{lem:equal}
Let $(X, \mu)$ be a standard probability space, let $\cF_1 \supseteq \cF_2$ be sub-$\sigma$-algebras, and let $\Sigma$ be a sub-$\sigma$-algebra with $\Sigma \vee \cF_2 = \Borel(X)$. If $\sH(\cP \given \cF_1) = \sH(\cP \given \cF_2)$ for all finite partitions $\cP \subseteq \Sigma$ then $\cF_1 = \cF_2$ mod null sets.
\end{lem}

\begin{proof}
Let $\cP \subseteq \Sigma$ and $\zeta \subseteq \cF_2$ be finite partitions and let $\alpha \leq \cP \vee \zeta$. Note that $\sH(\alpha \given \cF_1) \leq \sH(\alpha \given \cF_2)$ and $\sH(\cP \vee \zeta \given \alpha \vee \cF_1) \leq \sH(\cP \vee \zeta \given \alpha \vee \cF_2)$. Also note that $\sH(\cP \vee \zeta \given \cF_1) = \sH(\cP \given \cF_1) = \sH(\cP \given \cF_2) = \sH(\cP \vee \zeta \given \cF_2)$. Since
$$\sH(\alpha \given \cF_1) + \sH(\cP \vee \zeta \given \alpha \vee \cF_1) = \sH(\cP \vee \zeta \given \cF_1) = \sH(\cP \vee \zeta \given \cF_2) = \sH(\alpha \given \cF_2) + \sH(\cP \vee \zeta \given \alpha \vee \cF_2),$$
we must have that $\sH(\alpha \given \cF_1) = \sH(\alpha \given \cF_2)$. Since $\Sigma \vee \cF_2 = \Borel(X)$, the equality $\sH(\alpha \given \cF_1) = \sH(\alpha \given \cF_2)$ holds for a $\dR_\mu$-dense set of $\alpha$. It therefore holds for all partitions $\alpha$ with $\sH(\alpha) < \infty$. Now, taking $\alpha = \{A, X \setminus A\}$ for $A \in \cF_1$ we have $0 = \sH(\alpha \given \cF_1) = \sH(\alpha \given \cF_2)$ which implies that $\alpha \subseteq \cF_2$ mod null sets.
\end{proof}

We now present the main theorem of this section.

\begin{thm} \label{thm:cpeprod}
Let $G \acts (X, \mu)$ be an aperiodic {\pmp} action of stabilizer type $\theta$, and let $\cF$ be a $G$-invariant sub-$\sigma$-algebra. If $G \acts (Y, \nu)$ is a {\pmp} action of stabilizer type $\theta$ which is weakly contained in all {\pmp} actions of stabilizer type $\theta$, then
$$\Pi(\mu \relprod{\theta} \nu \given \cF \vee \Borel(Y)) = \Pi(\mu \given \cF) \vee \Borel(Y).$$
\end{thm}

\begin{proof}
Clearly $\Pi(\mu \given \cF) \vee \Borel(Y) \subseteq \Pi(\mu \relprod{\theta} \nu \given \cF \vee \Borel(Y))$. Fix a finite partition $\cP$ of $X$. For each $n \in \N$, apply Theorem \ref{thm:robin} to obtain a $G$-invariant sub-$\sigma$-algebra $\Psi_n \subseteq \Borel(X)$ such that $\rh_{G,\mu}(\Psi_n) < 1 / 2^n$ and such that the factor map associated to $\Psi_n$ preserves stabilizers. For each $n \in \N$ set
$$\Phi_n = \Pi \left( \mu \Given \bigvee_{k \geq n} \Psi_k \vee \cF \right) \supseteq \Pi(\mu \given \cF).$$
We have
$$\rh_{G,\mu} \left( \bigcap_{n \in \N} \Phi_n \Given \cF \right) \leq \inf_{n \in \N} \rh_{G,\mu}(\Phi_n \given \cF) \leq \inf_{n \in \N} \sum_{k \geq n} \rh_{G,\mu}(\Psi_k) = 0.$$
Therefore $\bigcap_{n \in \N} \Phi_n = \Pi(\mu \given \cF)$ and hence $\sH_\mu(\cP \given \Pi(\mu \given \cF)) = \lim_{n \rightarrow \infty} \sH_\mu( \cP \given \Phi_n)$. Fix $\delta > 0$ and fix $n \in \N$ with
$$\sH_\mu(\cP \given \Phi_n) > \sH_\mu(\cP \given \Pi(\mu \given \cF)) - \delta.$$

Let $G \acts (Z, \eta)$ be the factor of $(X, \mu)$ associated to $\Phi_n$. Note that the map $f : X \rightarrow Z$ preserves stabilizers by construction of $\Phi_n$. The factor map $f$ naturally produces a joining $\lambda = (\id \times f)_*(\mu)$. Note that $G \acts (X \times Z, \lambda)$ is isomorphic to $G \acts (X, \mu)$. By Corollary \ref{cor:awstab} $G \acts (X \times Z, \lambda)$ weakly contains $G \acts (X \times Y, \mu \relprod{\theta} \nu)$ as joinings with $G \acts (X, \mu)$. From our construction we have $\Pi(\lambda \given \cF \vee \Borel(Z)) = \Pi(\mu \given \Phi_n) = \Phi_n$. Lemma \ref{lem:cpejoin} gives 
\begin{align*}
\sH_{\mu \relprod{\theta} \nu}(\cP \given \Pi(\mu \given \cF) \vee \Borel(Y)) - \delta & \leq \sH_\mu(\cP \given \Pi(\mu \given \cF)) - \delta\\
 & < \sH_\mu(\cP \given \Phi_n)\\
 & = \sH_{\lambda}(\cP \given \Pi(\lambda \given \cF \vee \Borel(Z)))\\
 & \leq \sH_{\mu \relprod{\theta} \nu}(\cP \given \Pi(\mu \relprod{\theta} \nu \given \cF \vee \Borel(Y)))\\
 & \leq \sH_{\mu \relprod{\theta} \nu}(\cP \given \Pi(\mu \given \cF) \vee \Borel(Y)).
\end{align*}
Letting $\delta$ tend to $0$, we find that
$$\sH_{\mu \relprod{\theta} \nu}(\cP \given \Pi(\mu \given \cF) \vee \Borel(Y)) = \sH_{\mu \relprod{\theta} \nu}(\cP \given \Pi(\mu \relprod{\theta} \nu \given \cF \vee \Borel(Y)))$$
for all finite partitions $\cP$ of $X$. Now Lemma \ref{lem:equal} implies that $\Pi(\mu \given \cF) \vee \Borel(Y) = \Pi(\mu \relprod{\theta} \nu \given \cF \vee \Borel(Y))$.
\end{proof}

In the case of free actions, we can slightly reduce the assumptions of the above theorem. Specifically, we do not need to assume that $G$ acts freely on $(Y, \nu)$.

\begin{cor}
Let $G$ be a countably infinite group, let $G \acts (X, \mu)$ be a free {\pmp} action, and let $\cF$ be a $G$-invariant sub-$\sigma$-algebra. If $G \acts (Y, \nu)$ is a {\pmp} action which is weakly contained in all free actions of $G$, then
$$\Pi(\mu \times \nu \given \cF \vee \Borel(Y)) = \Pi(\mu \given \cF) \vee \Borel(Y).$$
\end{cor}

\begin{proof}
Let $G \acts (Y', \nu')$ be a free {\pmp} action which is weakly contained in all free actions of $G$ (such as a Bernoulli shift). By the previous theorem $\Pi(\mu \given \cF) \vee \Borel(Y') = \Pi(\mu \times \nu' \given \cF \vee \Borel(Y'))$. We have that $G \acts (X \times Y', \mu \times \nu')$ weakly contains $G \acts (X \times Y, \mu \times \nu)$ as joinings with $G \acts (X, \mu)$, so for every finite partition $\cP$ of $X$ Lemma \ref{lem:cpejoin} implies that
\begin{align*}
\sH_\mu(\cP \given \Pi(\mu \given \cF)) & = \sH_{\mu \times \nu'}(\cP \given \Pi(\mu \given \cF) \vee \Borel(Y'))\\
 & = \sH_{\mu \times \nu'}(\cP \given \Pi(\mu \times \nu' \given \cF \vee \Borel(Y')))\\
 & \leq \sH_{\mu \times \nu}(\cP \given \Pi(\mu \times \nu \given \cF \vee \Borel(Y)))\\
 & \leq \sH_{\mu \times \nu}(\cP \given \Pi(\mu \given \cF) \vee \Borel(Y))\\
 & = \sH_\mu(\cP \given \Pi(\mu \given \cF)).
\end{align*}
So $\sH_{\mu \times \nu}(\cP \given \Pi(\mu \given \cF) \vee \Borel(Y)) = \sH_{\mu \times \nu}(\cP \given \Pi(\mu \times \nu \given \cF \vee \Borel(Y)))$ for all finite partitions $\cP$ of $X$. Now apply Lemma \ref{lem:equal}.
\end{proof}

Finally, we obtain the strongest conclusion by placing weak containment assumptions on both $G \acts (X, \mu)$ and $G \acts (Y, \nu)$.

\begin{cor} \label{cor:pinsk2}
Let $G \acts (X, \mu)$ and $G \acts (Y, \nu)$ be aperiodic {\pmp} actions of stabilizer type $\theta$ which are weakly contained in all {\pmp} actions of stabilizer type $\theta$. Let $\cF$ and $\Sigma$ be $G$-invariant sub-$\sigma$-algebras of $X$ and $Y$, respectively. Then
$$\Pi(\mu \relprod{\theta} \nu \given \cF \vee \Sigma) = \Pi(\mu \given \cF) \vee \Pi(\nu \given \Sigma).$$
\end{cor}

\begin{proof}
Set $\Pi_X = \Pi(\mu \given \cF)$, $\Pi_Y = \Pi(\nu \given \Sigma)$, and $\Pi_{X \times Y} = \Pi(\mu \relprod{\theta} \nu \given \cF \vee \Sigma)$. Clearly $\Pi_X \vee \Pi_Y \subseteq \Pi_{X \times Y}$. By Theorem \ref{thm:cpeprod} $\Pi_{X \times Y} \subseteq \Pi(\mu \relprod{\theta} \nu \given \cF \vee \Borel(Y)) = \Pi_X \vee \Borel(Y)$. Fix finite partitions $\cP$ and $\cQ$ of $X$ and $Y$, respectively. By looking back near the end of the proof of Theorem \ref{thm:cpeprod}, we see that
$$\sH_\mu(\cP \given \Pi_X) = \sH_{\mu \relprod{\theta} \nu}(\cP \given \Pi_X) = \sH_{\mu \relprod{\theta} \nu}(\cP \given \Pi_X \vee \Borel(Y)).$$
In particular, since $\Pi_{X \times Y} \vee \cQ$ and $\Pi_X \vee \Pi_Y \vee \cQ$ each contain $\Pi_X$ and are contained in $\Pi_X \vee \Borel(Y)$, monotonicity properties of Shannon entropy imply that
$$\sH_{\mu \relprod{\theta} \nu}(\cP \given \cQ \vee \Pi_{X \times Y}) = \sH_{\mu \relprod{\theta} \nu}(\cP \given \Pi_X \vee \Borel(Y)) = \sH_{\mu \relprod{\theta} \nu}(\cP \given \cQ \vee \Pi_X \vee \Pi_Y).$$
By reversing the roles of $X$ and $Y$, we also get
$$\sH_{\mu \relprod{\theta} \nu}(\cQ \given \Pi_{X \times Y}) = \sH_{\mu \relprod{\theta} \nu}(\cQ \given \Pi_X \vee \Pi_Y).$$
Therefore $\sH_{\mu \relprod{\theta} \nu}(\cP \vee \cQ \given \Pi_{X \times Y}) = \sH_{\mu \relprod{\theta} \nu}(\cP \vee \cQ \given \Pi_X \vee \Pi_Y)$. This holds for all finite partitions $\cP$ and $\cQ$ of $X$ and $Y$, respectively. We can now repeat the argument in the proof of Lemma \ref{lem:equal} in order to conclude $\Pi_{X \times Y} = \Pi_X \vee \Pi_Y$.
\end{proof}


\thebibliography{999}


\bibitem{AGV}
M. Ab\'{e}rt, Y. Glasner, and B. Vir\'{a}g,
\textit{Kesten's theorem for invariant random subgroups}, Duke Mathematical Journal 163 (2014), no. 3, 465--488.

\bibitem{AW13}
M. Ab\'{e}rt and B. Weiss,
\textit{Bernoulli actions are weakly contained in any free action}, Ergodic Theory and Dynamical Systems 23 (2013), no. 2, 323--333.

\bibitem{Ag13}
I. Agol,
\textit{The virtual Haken conjecture}, Doc. Math., 18 (2013), 1045--1087. With an appendix by I. Agol, D. Groves, and J. Manning.

\bibitem{AS}
A. Alpeev and B. Seward,
\textit{Krieger's finite generator theorem for actions of countable groups III}, preprint. https://arxiv.org/abs/1705.09707.

\bibitem{A15}
T. Austin,
\textit{Additivity properties of sofic entropy and measures on model spaces}, preprint. http://arxiv.org/abs/1510.02392.

\bibitem{B10b}
L. Bowen,
\textit{Measure conjugacy invariants for actions of countable sofic groups}, Journal of the American Mathematical Society 23 (2010), 217--245.

\bibitem{B12}
L. Bowen,
\textit{Sofic entropy and amenable groups}, Ergod. Th. \& Dynam. Sys. 32 (2012), no. 2, 427--466.

\bibitem{B12b}
L. Bowen,
\textit{Every countably infinite group is almost Ornstein}, Dynamical systems and group actions, 67--78, Contemp. Math., 567, Amer. Math. Soc., Providence, RI, 2012.

\bibitem{BTD}
L. Bowen and R. Tucker-Drob,
\textit{On a co-induction question of Kechris}, Israel Journal of Mathematics 194 (2013), no. 1, 209--224.

\bibitem{Bu15}
P. Burton,
\textit{Naive entropy of dynamical systems}, preprint. http://arxiv.org/abs/1503.06360v2.

\bibitem{Da01}
A. I. Danilenko,
\textit{Entropy theory from the orbital point of view}, Monatsh. Math. 134 (2001), 121--141.

\bibitem{DGRS}
A. Dooley, V. Golodets, D. Rudolph, and S. Sinel'shchikov,
\textit{Non-Bernoulli systems with completely positive entropy}, Ergodic Theory and Dynamical Systems 28 (2008), no. 1, 87--124.

\bibitem{Do11}
T. Downarowicz,
Entropy in Dynamical Systems. Cambridge University Press, New York, 2011.

\bibitem{DFR}
T. Downarowicz, B. Frej, and P.-P. Romagnoli,
\textit{Shearer's inequality and infimum rule for Shannon entropy and topological entropy}, preprint. http://arxiv.org/abs/1502.07459.

\bibitem{GS15}
D. Gaboriau and B. Seward,
\textit{Cost, $\ell^2$-Betti numbers, and the sofic entropy of some algebraic actions}, preprint. http://arxiv.org/abs/1509.02482.

\bibitem{GTW}
E. Glasner, J.-P. Thouvenot, B. Weiss,
\textit{Entropy theory without past}, Ergodic Theory and Dynamical Systems 20 (2000), no. 5, 1355--1370.

\bibitem{H15}
B. Hayes,
\textit{Mixing and spectral gap relative to Pinsker factors for sofic groups}, preprint. http://arxiv.org/abs/1509.07839.

\bibitem{K95}
A. Kechris,
Classical Descriptive Set Theory. Springer-Verlag, New York, 1995.

\bibitem{K12}
A. Kechris,
\textit{Weak containment in the space of actions of a free group}, Israel Journal of Mathematics 189 (2012), 461--507.

\bibitem{KST99}
A. Kechris, S. Solecki, and S. Todorcevic,
\textit{Borel chromatic numbers}, Adv. in Math. 141 (1999), 1--44.

\bibitem{KL11a}
D. Kerr and H. Li,
\textit{Entropy and the variational principle for actions of sofic groups}, Invent. Math. 186 (2011), 501--558.

\bibitem{KL13}
D. Kerr and H. Li,
\textit{Soficity, amenability, and dynamical entropy}, American Journal of Mathematics 135 (2013), 721--761.

\bibitem{KL11b}
D. Kerr and H. Li,
\textit{Bernoulli actions and infinite entropy}, Groups Geom. Dyn. 5 (2011), 663--672.

\bibitem{Ki}
J. C. Kieffer,
\textit{A generalized Shannon--McMillan Theorem for the action of an amenable group on a probability space}, Annals of Probability 3 (1975), no. 6, 1031--1037

\bibitem{LS04}
A. Lubotzky and Y. Shalom,
\textit{Finite representations in the unitary dual and Ramanujan groups}, Contemp. Math. 347 (2004), 173--189.


\bibitem{OW87}
D. Ornstein and B. Weiss,
\textit{Entropy and isomorphism theorems for actions of amenable groups}, Journal d'Analyse Math\'{e}matique 48 (1987), 1--141.

\bibitem{Pe08}
V. Pestov,
\textit{Hyperlinear and sofic groups: a brief guide}, Bull. Symbolic Logic 14 (2008), no. 4, 449--480.

\bibitem{Ro67}
V. A. Rokhlin,
\textit{Lectures on the entropy theory of transformations with invariant measure}, Uspehi Mat. Nauk 22 (1967), no. 5, 3--56.

\bibitem{S12a}
B. Seward,
\textit{A subgroup formula for f-invariant entropy}, Ergodic Theory and Dynamical Systems 34 (2014), no. 1, 263--298.

\bibitem{S14}
B. Seward,
\textit{Krieger's finite generator theorem for actions of countable groups I}, preprint. http://arxiv.org/abs/1405.3604.

\bibitem{S14a}
B. Seward,
\textit{Krieger's finite generator theorem for actions of countable groups II}, preprint. https://arxiv.org/abs/1501.03367.

\bibitem{ST14}
B. Seward and R. D. Tucker-Drob,
\textit{Borel structurability on the $2$-shift of a countable group}, preprint. http://arxiv.org/abs/1402.4184.

\bibitem{Th15}
A. Thom,
\textit{A remark about the spectral radius}, Int. Math. Res. Not. (2015), no. 10, 2856--2864.

\bibitem{Th}
A. Thom,
\textit{The expected degree of minimal spanning forests}, to appear in Combinatorica.

\bibitem{TD12}
R. Tucker-Drob,
\textit{Weak equivalence and non-classifiability of measure preserving actions}, to appear in Ergodic Theory and Dynamical Systems. http://arxiv.org/abs/1202.3101.
\end{document}